\documentclass[a4paper,12pt]{amsart}

\usepackage{latexsym,amsthm,amsmath,amssymb}
\usepackage{color,ulem}
\usepackage{arydshln}
\usepackage{epic,eepic,eepicemu}
\usepackage{multicol}
\usepackage{colortbl}
\usepackage{graphicx}
\usepackage{pgfplots}
\pgfplotsset{compat=newest}

\usepackage{hyperref} % set hyperlinks

\renewcommand{\em}{\it}

\newtheorem{thm}{Theorem}[section]
\newtheorem{prop}[thm]{Proposition}
\newtheorem{lemma2}[thm]{Lemma}
\newtheorem{cor}[thm]{Corollary}

\theoremstyle{definition}

\newtheorem{exmpl}[thm]{Example}
\newtheorem{remark}[thm]{Remark}

\newcommand{\cL}{\mathcal{L}}

\newcommand{\fa}{\mathfrak{a}}
\newcommand{\fm}{\mathfrak{m}}

\newcommand{\B}{\mathbb{B}}
\newcommand{\R}{\mathbb{R}}
\newcommand{\N}{\mathbb{N}}

\renewcommand{\phi}{\varphi}
\renewcommand{\epsilon}{\varepsilon}

\newcommand{\bx}{\bar{x}}
\newcommand{\by}{\bar{y}}

\newcommand{\bs}{\bar{s}}
\newcommand{\bu}{\bar{u}}

\newcommand{\bA}{\overline{A}}

\newcommand{\tf}{\tilde{f}}

\renewcommand{\tt}{\tilde{t}}

\newcommand{\Span}{\mathop\mathrm{Span}}

\newcommand{\supp}{\mathop\mathrm{supp}}

\title[Exact Convergence Rates of Alternating Projections]{Exact Convergence Rates of Alternating Projections for Nontransversal Intersections}

\author[H. Ochiai]{Hiroyuki Ochiai}
\address[H. Ochiai]{Institute of Mathematics for Industry, Kyushu University, 744 Motooka, Nishi-ku, Fukuoka 819-0395, Japan}
\email{\texttt{ochiai@imi.kyushu-u.ac.jp}}
\author[Y. Sekiguchi]{Yoshiyuki Sekiguchi}
\address[Y. Sekiguchi]{Graduate School of Marine Science and Technology, Etchujima 2-1-8, Koto-ku, Tokyo 135-8533, Japan}
\email{\texttt{yoshi-s@kaiyodai.ac.jp}}
\author[H. Waki]{Hayato Waki}
\address[H. Waki]{Institute of Mathematics for Industry, Kyushu University, 744 Motooka, Nishi-ku, Fukuoka 819-0395, JAPAN}
\email{waki@imi.kyushu-u.ac.jp}


\begin{document}

\maketitle

%\tableofcontents

 \begin{abstract}
 We consider the convergence rate of the alternating projection method
 for the nontransversal intersection of a semialgebraic set and a linear
 subspace.
 For such an intersection, the convergence rate is known as sublinear in the worst case.
 We study the exact convergence rate for a given 
  semialgebraic set and an initial point, and investigate when the convergence rate is linear or sublinear.
 As a consequence, we show that the exact rates are expressed by
 multiplicities of
 the defining polynomials of the semialgebraic set, or related power series
 in the case that the linear subspace is a line,
 and we also decide the convergence rate for given data by using elimination theory.
 Our methods are also applied to give upper bounds for the case that the
 linear subspace has the dimension more than one.
 The upper bounds are shown to be tight by obtaining exact convergence
 rates for a specific semialgebraic set, which depend on the initial points.    
 \end{abstract}

\section{Introduction}
Convergence rates of iterative methods for optimization problems are
typically estimated in the worst case among optimization problems of a
specific type
and for any initial points.
In a practical application, such an estimate gives us useful information
for choosing an appropriate iterative method for the working problem.
However, the behavior of iterative methods certainly depends on
the input functions and the initial point, and the behavior sometimes changes
dramatically.

We are interested in the behavior of the alternating projection method
that strongly depends on given data and the initial point. 
The alternating projection method is an algorithm for finding
a point in the intersection of two sets, by iteratively projecting
points to each of the two sets.
The method has a variety of applications, such as  image recovery \cite{C1996}, \cite{BB1996},
phase retrieval \cite{BCL2002}, control theory \cite{GS1996} and factorization of completely positive matrices \cite{GD2020}.
In general, if two sets are semialgebraic, \cite{BLY2014} showed that the sequence constructed
by the alternating projections converges to a point in the intersection
without any regularity conditions.
If two closed convex sets intersect transversely, then
the convergence rate is linear \cite{GPR1967}, and
the behavior of alternating projections is well-known; see, e.g. \cite{LM2008}.
However, if the intersection is nontransversal,
then the convergence rate is sublinear, and the known upper bounds
on the convergence rate are far from being tight as discussed in
\cite[Remark 4.5]{BLY2014}. 
The convergence rate of alternating projections for a nontransversal intersection is also studied in \cite{DLW2017}, using H\"older regularity.

In this paper, we consider the \textit{exact convergence rate} of the sequence $\{u_k\}$
constructed by the alternating projection method;
\begin{equation}
 u_{k+1} = P_B\circ P_A(u_k),
  \label{eq:AP}
\end{equation}
where $P_A$ and $P_B$ are the projections onto convex sets $A$ and $B$ in $\R^n$, respectively.
To argue exact convergence rates, we restrict ourselves to the case where
$A$ is a semialgebraic convex set defined by one or two polynomials, $B$ is a linear subspace,
and the intersection $A\cap B$ is nontransversal and a singleton.
When $B$ is a line, we can
 directly analyze the recursive equations defining the sequence
 $\{u_k\}$ in \eqref{eq:AP}, and obtain the exact convergence rate.
  Namely, under the conditions of Theorem \ref{thm:hypersurface},
there exist $\lambda, C>0$ such that for any $\epsilon >0$,
\[
 (C - \epsilon)\frac{1}{k^\lambda} \leq \|u_k - \bu\|
 \leq (C + \epsilon)\frac{1}{k^\lambda}
\]
for sufficiently large $k$, where $\bu = \lim_{k\to\infty}u_k$.
Thus, we obtain both of an upper bound and a \textit{lower bound} on the
convergence rate. Moreover, we show that both bounds have
asymptotically the same degree and constant.
By applying Theorem \ref{thm:hypersurface} to the case where $A$ is defined by two polynomials, we can also determine the exact rate from the \textit{initial point} (Example \ref{ex4.12}).
Since one can rarely determine the exact convergence rate of an iterative method for optimization problems, this is a remarkable property of alternating projections.
Our results also improve corresponding estimates of the upper bounds on the
 convergence rate in \cite{BLY2014} to our setting while showing the obtained estimates are tight.
When $B$ has the dimension more than one,
then the exact convergence rate depends on the initial point even in the case where $A$ is defined by a single polynomial,
 and it seems to be hard to determine the exact rate for a general case as discussed in Section
 \ref{section:special}.

When the semialgebraic set $A$ is defined by a single polynomial
inequality, the recursive equations are analyzed
rigorously by using ideals of the ring of
convergent power series.
Then we show that the exact rate is determined by the multiplicity of the
defining polynomial of $A$ at the intersecting point (Theorem
$\ref{thm:hypersurface}$).
When $A$ is defined by two polynomial inequalities
and is in the three-dimensional space,
the boundary of $A$ is partitioned into three regions; two surfaces defined by each
polynomial and a curve defined by the two polynomials.
Then we obtain the exact rate
by using a number that can be
seen as a multiplicity of the curve defined by the two defining
polynomials of $A$
at the intersecting point if a point $b$ on the line $B$ is projected to the curve (Theorem \ref{thm:rate2poly}). 
We also give sufficient conditions that the projection $P_A(b)$ is on the curve for each $b$ on $B$ sufficiently close to the intersecting point (Theorem $\ref{thm:cond2poly}$).
Moreover, we show that the tangent plane to $A$ at the common point of
$A$ and $B$ is explicitly partitioned
into three regions; each of the two regions is projected to
the hypersurface defined by one of the two polynomials, and
the other region is projected to the curve defined by the two
polynomials (Theorem $\ref{thm:region})$.
This partition is calculated by the elimination theory for given polynomials
and is used to obtain the exact rates that depend on the initial points (Example
$\ref{ex4.12}, \ref{ex4.13}$).

The arguments on the exact rates are then applied to obtain upper bounds
of the rate
for the case that $A$ is defined by a single polynomial inequality and
$B$ is a linear subspace with the dimension more than one.
For general cases, we use the \L{}ojasiewicz exponent of the defining
polynomial of $A$ (Theorem $\ref{thm:hshp})$, or that of the restriction of the polynomial
to the linear subspace (Theorem $\ref{thm:loja})$, and give upper bounds.
Furthermore, for a specific polynomial, we obtain the exact rate
which depends on the initial point of the alternating projection method (Proposition $\ref{prop:special})$.
This specific case also shows that our upper bounds are tight.

The organization of the paper is the following. The basic notation and brief
explanations on a projection onto a convex set and on the analytic
implicit function theorem are given in Section $\ref{sec:prelim}$.
In Section $\ref{sec:recursive}$, we obtain the convergence rate of the
sequence defined by a special kind of a recursive equation, or inequality.
In Section $\ref{sec:line}$, we obtain exact convergence rates for
intersections of semialgebraic sets and lines.
Lastly, we give upper bounds for intersections of semialgebraic sets and
subspaces with dimensions more than one in Section $\ref{sec:subsp}$.

\subsubsection*{Related Work}
The alternating projection method has been extensively studied with notions of generalized regularity properties of intersections, such as metric regularity, metric subregularity, transversality,  subtransversality, H\"older regularity; see, e.g. a short survey in \cite[Section 2]{KLT2018} and \cite{DLW2017}.
These studies have built  a rich theoretical foundation on regularity  theory and the worst-case convergence analysis of iterative methods, 
and enable us to analyse far more general settings than traditional ones.
For relations between metric regularity and convergence analysis of iterative methods, see \cite{LTT2018} and references therein.

On the contrary, in this paper, we consider the exact convergence rate for a given instance from a special class of sets and intersections, for which the exact convergence rate of alternating projections can be obtained.
%On the contrary, we consider a special class of sets and intersections, for which the exact convergence rate of alternating projections can be obtained.
%
In the regularity studies above, the convergence analysis typically uses error bound-type inequalities and their quantitative information to estimate an upper bound on the convergence rate.
However, an estimate on the exact rate requires an estimate on a lower bound on the convergence rate. 
For this purpose, we directly analyze the recurrence equation that defines the sequence constructed by alternating projections, instead of using error bound inequalities.

Note that the intersections considered in this paper are not subtransversal as we can see that the sum rule of the normal cones does not hold at the intersecting point in Example \ref{ex:basic}; see, e.g. \cite[Proposition 5]{KLT2018}. 
By \cite[Corollary 3.4]{BLY2014}, the intersections considered in this paper are subtransversal with a gauge function \cite{LTT2018}, which is much weaker regularity than usual subtransversality.
% In addition, since subtransversality with a gauage function uses an arbitrary set containing the reference point instead of an open neighborhood, 
%  it extends applicability of metric regularity and enables us to
%   specify the direction of perturbations that are used in the
%   error bound inequality.
However, a lower bound on the convergence rate does not seem to be given by the gauge function, since it quantifies regularity via an upper error bound inequality. Neither does the exact rate since it depends on the initial point in general even in the case where $B$ is a line (Example \ref{ex4.12}).
In addition, the gauge function obtained by \cite[Corollary 3.4]{BLY2014} has an exponent which depends on the number of variables and on the maximum degree of constraint polynomials. Thus the upper bound on the convergence rate given in \cite{BLY2014} has a significant gap with the exact rate obtained in this paper which is independent of the number of variables and of the maximum degrees of constraint polynomials.

\section{Preliminaries}
\label{sec:prelim}
\subsection{Notation and Definitions}
Let $\|\cdot\|$ be the Euclidean norm on $\R^n$, $\langle x,y\rangle
= \sum_{i=1}^n x_iy_i$ for $x,y\in \R^n$, and $[n]=\{1,\ldots,n\}$.
Let $\partial A$ denote the boundary of a set $A\subset \R^n$.
The distance $d(x,A)$ from a point $x\in \R^n$ to a set $A\subset \R^n$ is defined by $d(x,A) = \inf_{a\in A}\|x - a\|$.

Let $\R[x]$ and $\R\{x\}$ be the set of polynomials
and the set of convergent power series in the variables $x =
(x_1,\ldots,x_n)$
with coefficients in $\R$, respectively.
For $f_1,\ldots,f_m \in \R\{x\}$, the
ideal generated by $f_1,\ldots,f_m$ is denoted by
$\langle f_1,\ldots,f_m\rangle$; i.e.
$\langle f_1,\ldots,f_m\rangle = \left\{\sum_{i=1}^m h_i f_i: h_i \in \R\{x\}\right\}$.
For $f\in \R\{x\}$,
the set of all the exponents
of the monomials appearing in $f$ is called the \textit{support} of $f$ and
denoted by $\supp f$.
The convex hull of the union $\bigcup_{\kappa\in \supp(f)}(\kappa +
\R^n_{\geq 0})$, which is  denoted by $\Gamma_+(f)$, is called the \textit{Newton diagram} of $f$,
and $\Gamma(f) = \bigcup(\text{compact face of }\Gamma_+(f))$ is called
the \textit{Newton boundary} of $f$.
For each face $\Delta
\in \Gamma(f)$, we define $f_\Delta(x)=\sum\{f_\alpha x^\alpha:\alpha\in \Delta\cap \supp f\}$.
A polynomial $f$ is said to be \textit{nondegenerate} if for each face $\Delta
\in \Gamma(f)$,
\[
 \frac{\partial f_\Delta}{\partial x_1}
 = \cdots =\frac{\partial f_\Delta}{\partial x_n}=0
\]
has no solution in $\left(\R\setminus\{0\}\right)^n$.

For $f, g:\R \to \R$, we write $f(x) = O(g(x))$ as $x\to \infty$
if there exist $C, M>0$ such that
$|f(x)| \leq C g(x)$ for all $x$ with $|x| > M$.
We also write $f(x) = \Theta(g(x))$ as $x\to \infty$
if there exist $C_1, C_2 > 0$ such that
$C_1 g(x) \leq f(x) \leq C_2 g(x)$ for all $x$ with $|x| > M$.
The meaning of the statement $f(x) = O(g(x))$ as $x\to 0$ is defined similarly.
If there is no ambiguity, we simply write $f(x) = O(g(x))$, or $f(x) = \Theta(g(x))$.
For a sequence $\{u_k\}\subset \R^n$ with $\bu = \lim_{k\to \infty}u_k$, we say that  $\{u_k\}$ converges 
in the rate $O(g(k))$
if $\|u_k - \bu\| = O(g(k))$ as $k\to \infty$, and 
in the \textit{exact rate}
$\Theta(g(k))$
if $\|u_k - \bu\| = \Theta(g(k))$ as $k\to \infty$.

\subsection{The projection and the implicit function theorem.}
\label{sec:proj_implicit}
We briefly review the projection and the implicit function theorem.
Let $B = \{x\in \R^{n}:f_i(x)\geq 0,i \in [m]\}$ for $f_i \in \R[x]$ for
$i \in [m]$. For $x\in B$, an index $i$ is said to be \textit{active} at
$x$ if $f_i(x) = 0$.
We say that $B$ is \textit{smooth} if
$\{\nabla f_i(x):i \text{ is active at }x\}$ is linearly independent for all $x\in B$.
In this paper, we define smoothness of $B$ for the particular defining polynomials of $B$. 
For a closed convex set $A\subset \R^n$ and $p \in \R^n$,
it is known that there exists
a unique optimal solution to
\begin{equation}
 \mathrm{minimize}
 \{\|x - p\|^2:x \in A\}.
 \label{min:projection}
\end{equation}
The optimal solution is called
the \textit{projection} of $p$ onto $A$ and denoted by $P_A(p)$.
  \begin{lemma2}
  \label{lemma:projection}
   Let $f_i \in
\R[x],\ i\in [m]$ and 
$A = \{x\in \R^{n}:f_i(x)\geq 0,\ i\in [m]\}$.
Suppose that $A$ is convex and $\{\nabla f_i(x):i \text{ is active at }x\}$ is linearly independent for all $x\in B.$ 
Then $u = P_A(p)$ if and only if there exist $c_i\in \R$ such that
\begin{equation}
 u - p = \sum_{i=1}^m c_i \nabla f_i(u),\ f_i(u) \geq 0,\ c_i\geq 0,\ c_i f_i(u) = 0,\ i\in [m].
 \label{eq:projection}
\end{equation}
  \end{lemma2}
\begin{proof}
 Since the objective function of $(\ref{min:projection})$ and $A$ are convex,
 we see that $u$ is optimal if and only if
 $-(u - p) \in N_A(u)$, where $N_A(u)$ is the normal cone of $A$ at
 $u$; i.e. $N_A(u)=\{y \in \R^n:\langle y, u' - u
 \rangle \leq 0, u'\in A\}$; 
 see, e.g. \cite{RW1998}.
 Then linear independence of $\{\nabla f_i(x):i \text{ is active at }x\}$ ensures that the equality
 \[
  N_A(u) = \left\{\sum_{i=1}^m c_i \nabla f_i(u):
 c_i \leq 0,\ c_i f_i(u)=0,\ i \in [m]\right\}
 \]
 holds \cite[Theorem 6.14]{RW1998}.
\end{proof}

Since we consider polynomial systems, the analytic implicit function
theorem ensures that the solution functions are
convergent power series; see, e.g. Theorem 6.1.2 and the following paragraph of \cite{KP2013}. 
 \begin{thm}
  [The implicit function theorem]
  \label{thm:implicit}
Let $m\leq n$ and $f_1,\ldots,f_m$ be polynomials in $\R[x,y] :=
\R[x_1,\ldots,x_m,y_1,\dots,y_{n-m}]$. Consider the system of equations
\[
 f_1(x,y) = \cdots = f_m(x,y) = 0.
\]
For a solution $(\bx,\by)$ to the system above, if $\{
\nabla f_{i,x}(\bx,\by):i\in [m]\}$ is linearly independent, then
 there exists a unique map $\phi(y)=(\phi_1(y),\ldots,\phi_m(y))$, where each $\phi_i(y)$ is a convergent power series 
around $\by$ such that
$\bx = \phi(\by)$ and $f_i(\phi(y),y)=0$ for $i\in [m]$ and $y$ close to $\by$.
\end{thm}
\begin{remark}
\label{remark:implicit}
 Let $n\leq m$ and $A = \{x\in \R^n:f_i(x) = 0,\ i \in [m]\}$.
 Suppose that the Jacobian matrix $\frac{\partial (f_1,\dots,f_m)}{\partial (x_1,\ldots,x_n)}(\bx)$ has  full rank at $\bx\in A$. 
 Then Theorem $\ref{thm:implicit}$ implies that for a tangent vector $v$ to $A$ at $\bx$, there exists a convergent power series $\phi(s)$ around $\bs$ such that $\phi(\bs) = \bx,\ \phi'(\bs)= v$ and $\phi(s)\in A$ for $s$ close to $\bs$; see, e.g. \cite[Exercise 6.7]{RW1998}. 
\end{remark}

\section{Recursive equation and inequality}
\label{sec:recursive}
The following lemma and corollary
give the convergence rate of the sequence defined by
recursive equation and inequality. They are
fundamental tools for our arguments and will be used repeatedly in the paper.
 \begin{lemma2}
  \label{lemma:basic}
 Suppose that the sequence $\{x_k\}$ satisfies $x_k>0$, $x_k \to 0$, and
 \[
  x_{k+1}\left(1 + Cx_{k+1}^{q} + x_{k+1}^{q+1}h(x_{k+1})\right) = x_k \ (k = 0, 1, \ldots),
 \]
 for some $C>0$, $q\in \N$ and a convergent power series $h(x)$.
  Then
  \[
   \lim_{k\to \infty}\left(qC\right)^{\frac{1}{q}}k^{\frac{1}{q}}x_k
 = 1.
  \]
 \end{lemma2}
 \begin{proof}
  First, we show that
  \[
   g(x) = \frac{\left(1 + Cx^q +
 x^{q+1}h(x)\right)^q - 1}{qCx^q}
  \]
  is a convergent power series around $x=0$, and $\displaystyle\lim_{x\to0} g(x) = 1$.
  In fact, we have
  \begin{align*}
   & qCx^q g(x) = \sum_{i=1}^q \binom{q}{i}(1 +
   Cx^q)^{q-i}\left(x^{q+1}h(x)\right)^i + (1 + Cx^q)^q - 1 \\
   & \phantom{qCx^q g(x)} = x^{q+1}h(x)\sum_{i=1}^q \binom{q}{i}(1 +
   Cx^q)^{q-i}\left(x^{q+1}h(x)\right)^{i-1} + \sum_{i=1}^q \binom{q}{i}(Cx^q)^i,\\
   & g(x) = \frac{xh(x)}{qC}\sum_{i=1}^q \binom{q}{i}(1 +
   Cx^q)^{q-i}\left(x^{q+1}h(x)\right)^{i-1} + \frac1q \sum_{i=2}^q
   \binom{q}{i}(Cx^q)^{i-1} + 1.
  \end{align*}
 Now we see that
\begin{align*}
   qCx_{k}^q & = qCx_{k+1}^q \left(1 + Cx_{k+1}^q +
 x_{k+1}^{q+1}h(x_{k+1})\right)^q \\
 &  = qCx_{k+1}^q \left(1 + qCx_{k+1}^q g(x_{k+1})\right),\\
 \frac{1}{qC x_{k+1}^q} - \frac{1}{qC x_k^q} & =
 \frac{1}{qC x_{k+1}^q} - \frac{1}{qC x_{k+1}^q \left(1 + qCx_{k+1}^q
 g(x_{k+1})\right)}\\
 & = \frac{g(x_{k+1})}{1 + qCx_{k+1}^q
 g(x_{k+1})}.
\end{align*}
  By summing the equation, we obtain
\[
  \frac{1}{qC x_{k}^q} - \frac{1}{qC x_0^q} 
  = \sum_{i=1}^k\frac{g(x_{i})}{1 + qCx_{i}^q
 g(x_i)},
\]
  and hence
\[
  \lim_{k\to \infty}\frac{1}{k qC x_{k}^q} 
  =  \lim_{k\to \infty}\frac{1}{k qC x_0^q}
  + \lim_{k\to \infty}\frac{1}{k}\sum_{i=1}^k\frac{g(x_{i})}{1 + qCx_{i}^q
 g(x_i)} = 1,
\]
 since the last summation is a Ces\`aro mean and $x_k\to 0$.
 \end{proof}

 \begin{cor}
  \label{cor:basic}
 Suppose that the sequence $\{x_k\}$ satisfies $x_k \geq 0$, $x_k\to 0$ and
 \[
  x_{k+1}\left(1 + Cx_{k+1}^{q} + x_{k+1}^{q+1}h(x_{k+1})\right) \leq x_k \ (k = 0, 1, \ldots),
 \]
 for some $C>0$, $q\in \N$ and a convergent power series $h(x)$.
 Then
 \[
  \limsup_{k\to \infty}\left(qC\right)^{\frac{1}{q}}k^{\frac{1}{q}}x_k
 \leq 1. 
 \]
 \end{cor}
 \begin{proof}
Since $x_k\to 0$, there exists $k_0$ such that 
  $1 + Cx_{k+1}^{q} +
  x_{k+1}^{q+1}h(x_{k+1})>0$ for $k \geq k_0$.
  If $x_k=0$ for some $k>k_0$, we have $x_{k+1} = 0$.
  Then the desired inequality holds.

  Thus, for each $k \geq k_0$, we assume $x_k>0$. By using $g(x)$ in Lemma
  \ref{lemma:basic}, we have
\begin{align*}
   qCx_{k}^q & \geq qCx_{k+1}^q \left(1 + Cx_{k+1}^q
 g(x_{k+1})\right),\\
 \frac{1}{qC x_{k+1}^q} - \frac{1}{qC x_k^q} & \geq
 \frac{1}{qC x_{k+1}^q} - \frac{1}{qC x_{k+1}^q \left(1 + qCx_{k+1}^q
 g(x_{k+1})\right)}\\
 & = \frac{g(x_{k+1})}{1 + qCx_{k+1}^q
 g(x_{k+1})} 
\end{align*}
  Since the limit of a Ces\`aro mean is not affected by an absence of the finite number of terms,
  the similar arguments in Lemma \ref{lemma:basic} implies that
  $\displaystyle\liminf_{k\to \infty} \frac{1}{k q C x_k^q}\geq 1$, and hence
  $\displaystyle\limsup_{k\to \infty} q C k x_k^q \leq 1$.
  Therefore, $\displaystyle\limsup_{k\to \infty} (q C)^\frac{1}{q} k^\frac{1}{q} x_k \leq 1$.
 \end{proof}
\begin{remark}
 If a sequence $\{x_k\}$ satisfies
 $\lim_{k\to\infty}Ck^\frac{1}{q}x_k=1$ for $C>0$,
 then, for any $\epsilon >0$, we have $(1 - \epsilon)C^{-1} k^{-\frac{1}{q}} \leq x_k
 \leq (1 + \epsilon)C^{-1} k^{-\frac{1}{q}}$ for sufficiently large $k$.
 Thus $x_k = \Theta(k^{-\frac{1}{q}})$.
 Therefore, it is implied that
 the condition $\lim_{k\to\infty}Ck^\frac{1}{q}x_k=1$ is a stronger property than the property $x_k = \Theta(k^{-\frac{1}{q}})$.
 Similarly, $\displaystyle\limsup_{k\to \infty} C k^\frac{1}{q} x_k \leq
 1$
 implies $x_k = O(k^{-\frac{1}{q}})$.
\end{remark}

 \section{Intersections with Lines}
 \label{sec:line}
 We consider the intersection of a semialgebraic convex set $A$ and a
linear subspace $B$. 
Let $P_A$ and $P_B$ be projections to $A$ and $B$ respectively.
We assume that the intersection is nontransversal and a singleton.
By translation, we also assume $A\cap B=\{0\}$. 
The alternating projection method constructs a sequence $\{u_k\}\subset
\R^{n+1}$ by \eqref{eq:AP}.
Note that $u_k$ converges to $0$; see, e.g. \cite[Fact 2.14]{BLY2014}.
In this section, we investigate the exact convergence rate of the alternating projection method in the case that $B$ is a line.
 
 \subsection{Hypersurfaces}
  \label{section:hypersurface}
  When we consider the alternating projection method for semialgebraic sets, only the boundaries have a crucial role.
  Thus, by an abuse of terminology, we call a semialgebraic set $A$ a \textit{hypersurface} if it is defined by a single polynomial.
In this section, we consider the case that $A$ is a hypersurface and $B$ is a line.
The following lemmas are stated in sufficient generality
that they can be used in later sections.
First, we give a characterization of the projection onto a hypersurface.
\begin{lemma2}
 \label{lemma:APsystem}
 For a nonnegative convex polynomial $g$, let
 $A = \{(x,z)\in \R^n\times \R:z \geq g(x)\}$.
For $(X,0)\notin A$, we have $(x,z) = P_A(X,0)$ if and only if
 the system
\[
 x_i + g_{x_i}(x)g(x) = X_i,\ i \in [n],\ z = g(x)
\]
 holds.
\end{lemma2}

\begin{proof}
Let $(X,0)\notin A$. Then the projection of $(X,0)$ onto $A$ lies on the boundary of $A$ by the definition.
Since $A$ is convex and $\nabla(z - g(x))$ is a nonzero 
vector for any $(x,z)$,
Lemma \ref{lemma:projection} implies that  $(x,z) = P_A(X,0)$ if and only if
 \[
  \begin{pmatrix}
   x\\
   z
  \end{pmatrix}
 - \begin{pmatrix}
    X \\
    0
   \end{pmatrix} = s\begin{pmatrix}
		     -\nabla g(x) \\
		     1
		    \end{pmatrix},\ z = g(x),
 \]
 for some $s\geq 0$. Since $s = z = g(x)$, we obtain the desired system.
\end{proof}
To analyze the equations in Lemma \ref{lemma:APsystem}, we need the following technical lemma for ideals.
 \begin{lemma2}
  \label{lemma:nakayama}
 Let $I, \fa, \fm$ be ideals in $\R\{x_1,\ldots,x_n\}$.
 If $I \subset \fa + \fm I$ and $\fm^s \subset \fa$ for some $s\in \N$,
 then $I \subset \fa$.
 \end{lemma2}
\begin{proof}
 We prove $I \subset \fa + \fm^k I$ for $k\in \N$ by induction. 
 Suppose $I \subset \fa + \fm^k I$. Then we have
 \[
 I \subset \fa + \fm^k(\fa + \fm I) = \fa + \fm^k \fa + \fm^{k+1} I
 = \fa + \fm^{k+1}I.
 \]
 Thus the claim is proved. Since $\fm^s \subset \fa$, we obtain $I
 \subset \fa$.
\end{proof}

By applying the lemmas above, we obtain an inclusion relation for ideals that gives a lower bound on the lowest degrees of convergent power series that solve recursive equations.
Note that for a polynomial $g(x)$ with $n$ variables, we say that $\Gamma(g)$ meets all the axes
if $g$ has a monomial $x_i^{d_i}$ with $d_i>0$ for
each $i=1,\dots,n$.

 \begin{lemma2}
  \label{lemma:newton}
For $1\leq m \leq n-1$, $(x,y)\in \R^m\times \R^{n-m}$ and 
a polynomial $g(x,y)$, we consider the system
\[
(*)\ x_{i} + g_{x_i}(x,y)
   g(x,y) = 0,\ i = 1,\ldots, m.
\]
  Suppose that $g(0,0)=g_{x_i}(0,0) = 0$ and $\Gamma(g(0,y))$ meets all the axes.
  Define ideals
  $\fm = \langle y_1,\dots,y_{n-m}\rangle$, $\fa = \langle
  y^\alpha:\alpha \in \supp(g(0,y))\rangle$,
  and $I=\langle \phi_1(y),\ldots, \phi_m(y)\rangle$ of $\R\{y\}$,
  where
  $\phi_i(y)$ is the convergent power series which solves $(*)$ as
  $x_i=\phi_i(y)$ and $\phi_i(0) = 0$ for  $i = 1,\ldots,m$.  
  Then we have $I \subset \fm\fa$.
 \end{lemma2}
\begin{proof}
Let $F=(x_i + g_{x_i}(x,y)g(x,y))_{i \in [m]}$.
  Since $g(0,0)=g_{x_i}(0,0)=0$, we see that $\left(\frac{\partial F_i}{\partial x_j}(0,0)\right)_{i,j}$ is the identity matrix. By the implicit function theorem (Theorem \ref{thm:implicit}), 
  there exist convergent power series $\phi_i(y)$
  which solve 
  the equation $(*)$ as $x_i = \phi_i(y)$ and
  $\phi_i(0) = 0$ for $i
  = 1, \ldots, m$.
  Let $\phi(y) = (\phi_1(y), \ldots, \phi_{n-r}(y))$.
  Then we have, for some polynomials $p_j$,
\begin{align*}
\phi_i(y) = x_i & = -g_{x_i}(\phi(y),y)g(\phi(y),y) \\
 & = - g_{x_i}(\phi(y),y)\left(g(0,y) +
 \sum\nolimits_{j=1}^m \phi_j(y) p_j(\phi(y),y)\right)\\
& = - g_{x_i}(\phi(y),y)g(0,y) 
 - g_{x_i}(\phi(y),y)\sum\nolimits_{j=1}^m \phi_j(y) p_j(\phi(y),y) 
\end{align*}
 Since $g_{x_i}(0,0) = 0$, we have $g_{x_i}(\phi(y),y) \in \fm$.
 % for all $i =  1, \ldots, n-r$.
  Thus the above equality implies that
  \[
   I \subset \fm \fa + \fm I.
  \]
Since $\Gamma(g(0,y))$ meets all the axes,
  we have $\fm^s \subset \fm \fa$ for $s = (n - m)(\deg g(0,y) + 1)$.
  By applying Lemma \ref{lemma:nakayama}, we obtain $I \subset \fm \fa$.
\end{proof}
Now, we state the main theorem in this section, which gives a formula for the exact rate using the multiplicity of a defining polynomial. For a convex polynomial $g$ and $a\in \R^n\setminus\{0\}$, let
\begin{align*}
 A & = \{(x,z) \in \R^n\times \R: z \geq g(x)\},\\
 B & = \{t(a,0) \in \R^n\times \R: t\in \R\},
\end{align*}
Suppose that $g(x)>0$ for $x\neq 0$ and $g(0)=0$.
   \begin{thm}
   \label{thm:hypersurface}
Suppose that $g(\|a\|^{-1}a t) = c_0 t^d + O(t^{d+1})$ for some $c_0,d>0$.
Then the sequence $\{u_k\}$ constructed by the alternating projection method $\eqref{eq:AP}$
   converges to $A\cap B$ with the exact rate $\Theta(k^{\frac{-1}{2d-2}})$.
   More precisely, we have
   \begin{equation}
    \lim_{k\to \infty}\left((2d-2)dc_0^2\right)^{\frac{1}{2d-2}}k^{\frac{1}{2d-2}}\|u_k\|
 = 1.
   \label{eq:hypersurface}
   \end{equation}
   \end{thm}
 \begin{proof}
By rotating about $z$ axis, we may assume that 
 $B = \{t(e_n,0)\in \R^n \times \R: t\in
 \R\}$, where $e_n=(0,\ldots,0,1)\in \R^n$.
 Then, we have
 \[
  g(0,x_n) = c_0 x_n^d + O(x_n^{d+1}).
 \]
  Since $g(x)>0$ for $x\neq 0$, we see $d\geq 2$.
  
 For a point $u_0=(te_n,0)$ of $B$,
 let $u = (x,z) = P_A(u_0)$ and $u_1 = (\tt e_n,0) = P_B(u)$.
 Then Lemma \ref{lemma:APsystem} implies that 
%\begin{align*}
%& u + s\left(\nabla g(x), - 1\right) = %u_0,\ z = g(x),\\
%& u_1 = ((0,x_n),0)
%\end{align*}
% for some $s\in \R$. Then $s = g(x)$ and we obtain
\begin{align}
&  x_i + g_{x_i}(x)g(x) = 0,\ i \in [n-1], \label{eq:ap1}\\
 & x_n + g_{x_n}(x)g(x) = t,\label{eq:ap2}\\
 & \tt = x_n.\notag
\end{align} 
By the implicit function theorem, there exist convergent power series
 $\phi_i(x_n)$ which solve equations $(\ref{eq:ap1})$ as $x_i =
 \phi_i(x_n)$
 with $\phi_i(0) = 0$ for $i = 1,\ldots,n-1$.
Here we note that $g(0) = 0$, $g_{x_i}(0)=0$ and $g(0,x_n)=c_0x_n^d + O(x_n^{d+1})$ with $c_0\neq 0$.
  Then we can apply Lemma \ref{lemma:newton} to the equation $(\ref{eq:ap1})$,
  where $r = 1$, $\fm = \langle x_n \rangle$,
  $\fa = \langle x_n^d \rangle$, and $I = \langle \phi_1(x_n), \ldots,
  \phi_{n-1}(x_n)\rangle$. Thus we obtain $\phi_i(x_n) \in \langle
  x_n^{d+1}\rangle$,
  which means $\phi_i(x_n) = O(x_n^{d+1})$ 
  for $i = 1, \ldots, n-1$.

 Next, we will modify the equation $(\ref{eq:ap2})$ to obtain a relation between $\|u_0\|$ and $\|u_1\|$. Let $\phi(x_n) = (\phi_1(x_n),
 \ldots, \phi_{n-1}(x_n))$.
 Now we have, for some polynomials $p_i, r_i$,
 \begin{align*}
  g(\phi(x_n),x_n) & = g(0,x_n) + \sum_{i=1}^{n-1} \phi_i(x_n) p_i(\phi(x_n),x_n)
  = c_0 x_n^d + O(x_n^{d+1}),\\
  g_{x_n}(\phi(x_n), x_n) & = g_{x_n}(0,x_n)
  + \sum_{i=1}^{n-1} \phi_i(x_n) r_i(\phi(x_n),x_n) = dc_0 x_n^{d-1} + O(x_n^d).
 \end{align*}
 Then the equation $(\ref{eq:ap2})$ gives that
\[
t= x_n + g_{x_n}(\phi(x_n),x_n)g(\phi(x_n),x_n) = x_n +  dc_0^2 x_n^{2d-1} + O(x_n^{2d}),
\]
 and hence
 \[
t=  \tt + dc_0^2 \tt^{2d-1} + O(\tt^{2d}). 
 \]
 Since $\|u_0\| =t,\ \|u_1\|= \tt$, 
 by repetedly applying the argument above, we have
 \[\|u_k\|=  \|u_{k+1}\| + dc_0^2 \|u_{k+1}\|^{2d-1} + O(\|u_{k+1}\|^{2d}).
 \]
 Here, we note that $O(\|u_{k+1}\|^{2d})$ 
 is the same convergent power series for each $k=0,1,\ldots$,
 since we always have equations $(\ref{eq:ap1}),\ (\ref{eq:ap2})$ in each iteration and $\|u_k\|$ is decreasing.
 Therefore Lemma $\ref{lemma:basic}$ implies
 the equality $(\ref{eq:hypersurface})$.
 \end{proof}
\begin{remark}
In general, for a univariate convergent power series $f(x)=cx^d + O(x^{d+1})$ with $c\neq 0$, the lowest degree $d$ of $f$ is called the multiplicity of $f$ at $0$.
\end{remark}
\begin{exmpl}
\label{ex:basic}
 Let $A = \{(x,y,z)\in \R^3:z \geq g(x,y)\}$ for $g(x,y) = x^2 + y^4$
 and $B = \{t(a,b,0) \in \R^3: t \in \R\}$.
 Define $\{u_k\}$ by $u_{k+1} = P_B \circ P_A(u_k)$.
 Then we have 
 \[
  g(\|(a,b)\|^{-1}(a,b)t) = \frac{a^2}{a^2 + b^2}t^2 + \frac{b^4}{(a^2 +
 b^2)^2}t ^4.
 \]
 Let $d$ be the lowest degree of the polynomial above and $c_0$ be its coefficient.
 If $a\neq 0$, then $d = 2$ and $c_0 = \frac{a^2}{a^2 + b^2}$.
 By Theorem \ref{thm:hypersurface}, we have
 $\lim\limits_{k\to\infty}\frac{2a^2}{a^2 + b^2}k^\frac{1}{2}\|u_k\|
 = 1$. On the other hand, if $a = 0$, then
 $d = 4$ and 
 $c_0 = 1$.
 We have
 $\lim\limits_{k\to\infty}\left(24\right)^\frac{1}{6}k^\frac{1}{6}\|u_k\|
 = 1$.
\end{exmpl}
\begin{remark}
We can easily extend Theorem \ref{thm:hypersurface} to the case that $A = \{x \in \R^{n+1}:f(x)\geq 0\}$ where $f \in \R[x]$ is nonsigular at the intersection point $0$.
To see this, let $P=(p_1 \cdots p_{n+1})$ be the orthogonal matrix where
$\{p_1,\dots, p_n\}$ is an orthogonal basis for the tangent plane to $A$ at $0$ which contains $B$, and $p_{n+1}$ is $\|\nabla f(0)\|^{-1}\nabla f(0)$. 
Let $\tf(x) = f(Px)$. 
Then $\nabla \tf(0) = (0,\ldots,0, p_{n+1}^T\nabla f(0))$,  
and nonsingularity of $f$ at $0$ implies $\tf_{x_{n+1}}(0)\neq 0$.
Thus the implicit function theorem $\ref{thm:implicit}$ implies that
there exists a convergent power series $g$ and an open neighborhood $U$ of $0$ such that
$A \cap U = \{(x,z) \in \R^n\times \R: z \geq g(x)\}$.
Then almost identical arguments of the proof hold for a convergent power series $g$. Similarly, we can extend our results in the later sections to a slightly more general set whose defining inequality is written as $f(x)\geq0$ for some $f \in \R[x]$.
However, we keep considering cases that a defining inequality is written as $z \geq g(x)$ for some $g\in \R[x]$, for the sake of simplicity.
\end{remark}

\subsection{Sets Defined by Two Polynomials}
We consider the case that $A, B\subset \R^3$ and that $A$ is defined by two
polynomials and $B$ is a line.
For convex polynomials $f_1,f_2$ and $(a,b)\neq (0,0)$, let
\begin{align*}
 A & = \{(x,y,z) \in \R^3: z \geq f_1(x,y), z\geq f_2(x,y)\},\\
 B & = \{t(a,b,0) \in \R^3: t\in \R\}.
\end{align*}
Suppose that the intersection of $A$ and $xy$-plane is $\{(0,0,0)\}$. We assume 
\[
 C  = \{(x,y,z) \in \R^3: z = f_1(x,y) = f_2(x,y))\}
\]
 is smooth in the sense of Section $\ref{sec:proj_implicit}$.
 In this section, we first obtain the exact convergence rate under the assumption that all points on $B$ which are sufficiently close to the origin are projected to $C$ by $P_A$ (Theorem \ref{thm:rate2poly}). Then we discuss 
 a sufficient condition that the assumption holds (Theorem \ref{thm:cond2poly}).

Let $(\alpha_1,\alpha_2,\alpha_3)$ be a nonzero tangent vector to $C$ at the origin, which is a generating vector of the kernel of the matrix $\left(\begin{smallmatrix}
	     -f_{1,x}(0,0) & -f_{1,y}(0,0) & 1 \\
	     -f_{2,x}(0,0) & -f_{2,y}(0,0) & 1 \\
	    \end{smallmatrix}\right)$. 
	    Since $xy$-plane
 is tangent to $C$ there, we see that $\alpha_3=0$. By Remark $\ref{remark:implicit}$, there exist convergent power series
$\phi_1(s),\phi_2(s),\phi_3(s)$ such that
\[
\phi(s)
 = \begin{pmatrix}
  \phi_1(s) \\
  \phi_2(s) \\
  \phi_3(s)
 \end{pmatrix}
 = \begin{pmatrix}
    \alpha_1 \\
    \alpha_2\\
    0
   \end{pmatrix}s
   + \begin{pmatrix}
      \beta_1\\
      \beta_2\\
      \beta_3
     \end{pmatrix}s^2 + O(s^3),
\]
and $C$ is the image of $\phi$ locally around the origin.
% Note that the kernel of the matrix
%  $\left(\begin{smallmatrix}
% 	     -f_{1,x}(0,0) & -f_{1,y}(0,0) & 1 \\
% 	     -f_{2,x}(0,0) & -f_{2,y}(0,0) & 1 \\
% 	    \end{smallmatrix}\right)$ 
% 	    includes the coefficient vector $(\alpha_1,\alpha_2,\alpha_3)$ of the linear term of $\phi(s)$ and that a generating vector $v=(v_1,v_2,v_3)$ of the kernel
% 	    is a tangent vector  of $C$ at
%  	    $(0,0,0)$.
%  Since $xy$-plane is tangent to $A$ at $(0,0,0)$,
%  we see that $v_3 = 0$.
%  Thus
%  $\left(\begin{smallmatrix}
% 	     -f_{1,x}(0,0) & -f_{1,y}(0,0)\\
% 	     -f_{2,x}(0,0) & -f_{2,y}(0,0)\\
%   \end{smallmatrix}\right)$
%     has a nontrivial solution.
%     Then
%  $\det\left(\begin{smallmatrix}
% 	     -f_{1,x}(0,0) & 1\\
% 	     -f_{2,x}(0,0) & 1\\
% 	     \end{smallmatrix}\right)$
% or $\det\left(\begin{smallmatrix}
% 	     -f_{1,x}(0,0) & 1\\
% 	     -f_{2,x}(0,0) & 1\\
% 	     \end{smallmatrix}\right)$ is nonzero, and hence
%     $C$	can be written as the graph of a function in $x$ or $y$ variable.
% Therefore, we may assume that $(\alpha_1,\alpha_2)\neq (0,0)$ and $\alpha_3 = 0$.

  \begin{thm}
  \label{thm:rate2poly}
   Let $B=\{t(\alpha_1,\alpha_2,0)\in \R^3:t\in \R\}$ and $\{u_k\}$ be the sequence constructed by the alternating projection method $(\ref{eq:AP})$.
   Suppose that $P_A(u_k) \in C$ for all sufficiently large $k$,
   and that $d$ is the lowest degree of 
   \begin{equation}
    \frac{1}{\sqrt{\alpha_1^2 + \alpha_2^2}}\left((\alpha_2\phi_1(s) -
   \alpha_1\phi_2(s))^2 + (\alpha_1^2 + \alpha_2^2)\phi_3(s)^2\right)^{\frac{1}{2}} = c_0s^d + O(s^{d+1}),\label{eq:rate2poly}
   \end{equation}
   for some $c_0>0$ as a power series in $s$.
   Then $\{u_k\}$ converges to $0$ in the exact rate
   $\Theta\left(k^{\frac{-1}{2d-2}}\right)$.
   Moreover, 
       \begin{equation}
\lim_{k\to \infty}\left(\frac{(2d-2)dc_0^2}{(\alpha_1^2 + \alpha_2^2)^d}\right)^{\frac{1}{2d-2}}k^{\frac{1}{2d-2}}\|u_k\|
 = 1.
 \label{eq:rate2poly2}
\end{equation}
  \end{thm}
   \begin{proof}
    First, we calculate $d(\phi(s),B)$. Let $t(\alpha_1,\alpha_2,0) =
    P_B(\phi(s))$.
    By the property of the projection, we have
\begin{align}
  &   \begin{pmatrix}
     \alpha_1 \\
     \alpha_2\\
     0
   \end{pmatrix}\cdot \phi(s)
   = \begin{pmatrix}
     \alpha_1 \\
     \alpha_2\\
     0
   \end{pmatrix}\cdot t\begin{pmatrix}
       \alpha_1\\
       \alpha_2\\
       0
		       \end{pmatrix},
 \label{eq:innerprod1}\\
    & t = \frac{1}{\alpha_1^2 + \alpha_2^2}(\alpha_1\phi_1(s) +
    \alpha_2\phi_2(s)).\notag 
\end{align}   
Thus we obtain
\begin{multline*}
     d(\phi(s), B)^2 = \|(\phi_1(s), \phi_2(s), \phi_3(s)) -
 t(\alpha_1,\alpha_2,0)\|^2 \\
 = \frac{1}{\alpha_1^2 + \alpha_2^2}
 \left( (\alpha_2 \phi_1(s) - \alpha_1 \phi_2(s))^2 + (\alpha_1^2 + \alpha_2^2)\phi_3(s)^2\right).
\end{multline*}
 By the equation $\eqref{eq:rate2poly}$, we see $d(\phi(s),B)=c_0 s^d + O(s^{d+1})$.
    
 Next, we  rotate
 the curve $C$ about $z$-axis and reparametrize its parameter by a nonzero scalar multiple
 so that
 $(\alpha_1,\alpha_2,0)=(1,0,0)$. 
  Then the curve $C$ is represented by
 $\psi(x)=(x,\psi_2(x), \psi_3(x))$ for some convergent power series $\psi_2,\psi_3$.
   Suppose that $Q\overset{P_A}{\longmapsto} R:=\psi(x)
   \overset{P_B}{\longmapsto} S$ is written as
   \[
   T\begin{pmatrix}
     1 \\
     0 \\
     0
   \end{pmatrix}\overset{P_A}{\longmapsto}
   \begin{pmatrix}
    x \\
    \psi_2(x) \\
    \psi_3(x)
   \end{pmatrix}\overset{P_B}{\longmapsto}
   t\begin{pmatrix}
     1 \\
     0 \\
     0
    \end{pmatrix},
   \]
   for some $x, t, T\in \R$.
   By applying equation $(\ref{eq:innerprod1})$,
   we have 
    $x = t$. Since $R$ is the projection of $Q$ onto
   $C$, we see that $\overrightarrow{RQ}$ is orthogonal to $C$.
   Thus we have
\begin{align}
 &   \left(
   \begin{pmatrix}
    x \\
    \psi_2(x) \\
    \psi_3(x)    
   \end{pmatrix}
   -T\begin{pmatrix}
     1\\
     0\\
     0
    \end{pmatrix}
   \right)\cdot
   \begin{pmatrix}
    1\\
    \psi'_2(x)\\
    \psi'_3(x)
 \end{pmatrix}= 0,\notag \\
 & t + \psi_2(t)\psi'_2(t) + \psi_3(t)\psi'_3(t) = T.
 \label{eq:innerprod2}
\end{align}
    Here, let $h(x)=\sqrt{\psi_2(x)^2 + \psi_3(x)^2}$.
    Since $P_B(\psi(x)) = (x,0,0)$, we have
    $d(\psi(x), B) = d(\psi(x),(x,0,0)) = h(x)$.
    By comparing the speed of the parametric curve $\phi(s)$ with that of $\psi(x)$, we see that 
    $
    d(\psi(x), B) = cx^d + O(x^{d+1})
    $, where $c=\frac{c_0}{(\alpha_1^2 + \alpha_2^2)^\frac{d}{2}}$.
   Now we have
   \[
 \frac{d}{dx}\left(\frac{1}{2}h(x)^2\right) = \psi_2(x)\psi'_2(x) + \psi_3(x)\psi'_3(x).
   \]
   Thus $\psi_2(t)\psi'_2(t) + \psi_3(t)\psi'_3(t)=dc^2t^{2d-1} +
   O(t^{2d})$.
   Let $u_k$ and $u_{k+1}$ be the coordinate vectors of $Q$ and $S$, respectively. Then the equation
    $(\ref{eq:innerprod2})$ can be written as
   \[
    \|u_{k+1}\| + dc^2 \|u_{k+1}\|^{2d-1} + O(\|u_{k+1}\|^{2d}) = \|u_k\|.
   \]
   Therefore, Lemma \ref{lemma:basic} implies the equality $(\ref{eq:rate2poly2})$.
   \end{proof}
\begin{exmpl}
\label{ex:region1}
 Let $A = \{(x,y,z)\in \R^3:z\geq f_1(x,y),\ z\geq f_2(x,y)\}$, where
 \[
  \begin{cases}
   f_1(x,y) = x^2 + y^4,\\
   f_2(x,y) = (x - 1)^2 + (y - 1)^4 - 2.
  \end{cases}
 \]
  The tangent line to the curve $C$ at the origin is given by
% $B = \{(x,y,0)\in \R^3:x + 2y = 0\}$.
  $B = \{t(-2,1,0)\in \R^3:t \in \R\}$.
 The curve $C$ is written as 
 \[
 \phi(y) 
 = \begin{pmatrix}
    -2y +3y^2 - 2y^3\\
    y\\
    4y^2 - 12y^3 + 18y^4 - 12y^5 + 4y^6
    \end{pmatrix}
 \]
%  We use new coordinates
%  $\left(\begin{smallmatrix}
% 	 \tx\\
% 	 \ty
%  \end{smallmatrix}\right) =
%  \frac{1}{\sqrt{5}}
%  \left(\begin{smallmatrix}
% 	-2 & 1 \\
% 	-1 & -2
%  \end{smallmatrix}\right)
%  \left(\begin{smallmatrix}
% 	x\\
% 	y
%       \end{smallmatrix}\right)
%  $,
%  and rewrite $\tx,\ty$ as $x,y$ again. Then we have 
%   $B = \{t(1,0,0)\in \R^3:t \in \R\}$ and 
%  \[
%   \begin{cases}
%   f_1(x,y) =
% \frac{4}{5} x^{2}+\frac{4}{5} x y+\frac{1}{5} y^{2}+\frac{1}{25} x^{4}-\frac{8}{25} x^{3} y+\frac{24}{25}
%   x^{2} y^{2}-\frac{32}{25} x y^{3}+\frac{16}{25} y^{4},\\
%   f_2(x,y) =
% 2 \sqrt{5} y+2 x^{2}-4 x y+5 y^{2}-\frac{4}{25} \sqrt{5} x^{3}+\frac{24}{25} \sqrt{5} x^{2}
%   y-\frac{48}{25} \sqrt{5} x 
%   y^{2}+\frac{32}{25} \sqrt{5} y^{3}\\
%   \phantom{f_2(x,y) = } 
%   +\frac{1}{25} x^{4}
%   -\frac{8}{25} x^{3} y+\frac{24}{25} x^{2}
%      y^{2}-\frac{32}{25} x y^{3}+\frac{16}{25} y^{4}.
%   \end{cases}
%  \]
%  Here, $C=\{(x,y,z)\in \R^3:z=f_1(x,y)=f_2(x,y)\}$ is written as
% 	  \[
% 	   \begin{pmatrix}
% 	    x\\
% 	    y\\
% 	    z
% 	   \end{pmatrix}
% 	  = \begin{pmatrix}
% 	     1 \\
% 	     0 \\
% 	     0
% 	    \end{pmatrix}x
% 	     + \begin{pmatrix}
% 		0\\
% 	     -\frac{3}{5\sqrt{5}}\\	     
% 		\frac{4}{5}
% 	       \end{pmatrix} x^2 + O(x^3).
% 	  \]
  By Theorem \ref{thm:cond2poly} below, we can easily check that
 $P_A(t(-2,1,0))\in C$ for all sufficiently small $t>0$; see Example \ref{ex:region2}. Now the equation $\eqref{eq:rate2poly}$ is $\sqrt{\frac{89}{5}}y^2 + O(y^3)$.
%  \[
%  \frac{1}{\sqrt{5}}\sqrt{89y^4 - 492y^5 + 1444y^6 - 2640y^7 + 3220y^8 - 2640y^9 + 1440y^{10} - 480y^{11} + 80y^{12}
%  \]
% \[
%   h(x) = \sqrt{y^2 + z^2} = \sqrt{\frac{9}{125}x^4 + \frac{16}{25}x^4
%  + O(x^5)} = \sqrt{\frac{89}{125}}x^2 + O(x^3). 
% \]
 Thus, for $u_{k+1} = P_B\circ P_A(u_k)$, Theorem \ref{thm:rate2poly} implies
       \[ 
\lim_{k\to \infty}\sqrt{\frac{356}{125}}k^{\frac{1}{2}}\|u_k\|
 = 1,
\] 
and hence $\|u_k\| = \Theta(k^{-\frac{1}{2}})$.
\end{exmpl}

  \begin{thm}
     \label{thm:cond2poly}
     Suppose that $(a,b,0)$, $\nabla(z -
   f_1)(0,0,0)$, $\nabla(z - f_2)(0,0,0)$ are linearly independent.
     \begin{enumerate}
    \item If $(a,b,0) = c(\alpha_1,\alpha_2,0)$ for some $c\neq
    0$ and the solution $(\lambda_1, \lambda_2, \mu)$ to the system
  \begin{equation}
   \begin{pmatrix}
    \beta_1\\
    \beta_2\\
    \beta_3
   \end{pmatrix}
  =\lambda_1\begin{pmatrix}
	     -f_{1x}(0,0)\\
	     -f_{1y}(0,0)\\
	     1
	    \end{pmatrix}
  +\lambda_2\begin{pmatrix}
	     -f_{2x}(0,0)\\
	     -f_{2y}(0,0)\\
	     1
	    \end{pmatrix}
  + \mu\begin{pmatrix}
	a\\
	b\\
	0
       \end{pmatrix}.
       \label{eq:cond2poly}
  \end{equation}
    satisfies $\lambda_1,\lambda_2>0$, then
    $P_A(p) \in C$ for each point $p$ in $B$ close to the origin.
   \item If $(a,b,0) \neq c(\alpha_1,\alpha_2,0)$ for any $c\neq 0$, or
   the solution $(\lambda_1,\lambda_2,\mu)$ to $(\ref{eq:cond2poly})$ 
   satisfies
   $\lambda_1\lambda_2<0$, then there exists $\epsilon>0$
  such that $P_A(p) \notin C$ for any point $p$ in $B\cap \epsilon\B\setminus\{0\}$.
   \end{enumerate}
  \end{thm}
  \begin{proof}
   Suppose that $(a,b,0)=c(\alpha_1, \alpha_2,0)$ for some $c\neq 0$
   and $(\beta_1,\beta_2,\beta_3)$ is written as $(\ref{eq:cond2poly})$.
   We will apply Lemma \ref{lemma:projection} to show that
   for each $t$ sufficiently close to $0$,
   there exist $s$ such that $P_A(t(a,b,0)) = (\phi_1(s),\phi_2(s),\phi_3(s))$.
  The equation in the condition of Lemma \ref{lemma:projection} can be written as
  \begin{align}
   &
   c_1\begin{pmatrix}
       -f_{1x}\\
       -f_{1y}\\
       1
      \end{pmatrix}
   + c_2\begin{pmatrix}
	 -f_{2x}\\
	 -f_{2y}\\
	 1
	\end{pmatrix}
   =\begin{pmatrix}
     \phi_1(s)\\
     \phi_2(s)\\
     \phi_3(s)
    \end{pmatrix}
   - t\begin{pmatrix}
      a\\
      b\\
      0
   \end{pmatrix},\notag\\
   &
   \begin{pmatrix}
    -f_{1x} & -f_{2x} & a \\
    -f_{1y} & -f_{2y} & b \\
    1 & 1 & 0
   \end{pmatrix}
   \begin{pmatrix}
    c_1 \\
    c_2 \\
    t
   \end{pmatrix}
   =\begin{pmatrix}
     \phi_1(s) \\
     \phi_2(s)\\
     \phi_3(s)
    \end{pmatrix},\label{eq:cond2poly2}
  \end{align}
   where $f_{ix} = f_{ix}(\phi(s)),\ f_{iy} =
   f_{iy}(\phi(s))$.
   We put 
   \[
      M =
   \begin{pmatrix}
    -f_{1x} & -f_{2x} & a \\
    -f_{1y} & -f_{2y} & b \\
    1 & 1 & 0
   \end{pmatrix},\quad   
   M_0 =
   \begin{pmatrix}
    -f_{1x}(0,0) & -f_{2x}(0,0) & a \\
    -f_{1y}(0,0) & -f_{2y}(0,0) & b \\
    1 & 1 & 0
   \end{pmatrix}.   
   \]
   Since the column vectors of $M_0$ are linearly independent,
   the linear equations \eqref{eq:cond2poly2} have a solution for $s$ close to $0$, and
   we have
 \begin{align}
  t & = \frac{1}{|M|}
  \begin{vmatrix}
    -f_{1x} & -f_{2x} & \phi_1(s) \\
    -f_{1y} & -f_{2y} & \phi_2(s) \\
    1 & 1 & \phi_3(s)   
  \end{vmatrix} \notag
  \\
  & = \frac{1}{|M_0|}
  \begin{vmatrix}
    -f_{1x}(0,0) & -f_{2x}(0,0) & \alpha_1 \\
    -f_{1y}(0,0) & -f_{2y}(0,0) & \alpha_2 \\
    1 & 1 & 0   
  \end{vmatrix}s + O(s^2) = c^{-1}s + O(s^2),\label{eq:cond2poly3}
  \\
      c_1 & = \frac{1}{|M|}
    \begin{vmatrix}
     \phi_1(s) & -f_{2x} & a \\
     \phi_2(s) & -f_{2y} & b \\
     \phi_3(s) & 1 & 0
    \end{vmatrix} \notag
    \\ 
 & = \frac{1}{|M|}
     \begin{vmatrix}
     \alpha_1 & -f_{2x} & a \\
     \alpha_2 & -f_{2y} & b \\
     0 & 1 & 0
     \end{vmatrix}s
 + \frac{1}{|M|}
     \begin{vmatrix}
     \beta_1 & -f_{2x} & a \\
     \beta_2 & -f_{2y} & b \\
     \beta_3 & 1 & 0
     \end{vmatrix}s^2 + O(s^3).\notag
 \end{align}
   Since the first term of $c_1$ is $0$, we have
   \[
    c_1 =  \frac{1}{|M_0|}
     \begin{vmatrix}
     \beta_1 & -f_{2x}(0,0) & a \\
     \beta_2 & -f_{2y}(0,0) & b \\
     \beta_3 & 1 & 0
     \end{vmatrix}s^2 + O(s^3).
   \]
   By the condition $(\ref{eq:cond2poly})$, we have
   $c_1 = \lambda_1 s^2 + O(s^3)$. Similarly, we have
   $c_2 = \lambda_2 s^2 + O(s^3)$.
   Thus, if $\lambda_1,\lambda_2>0$, then, for $t$ sufficiently close to $0$, there exists
   $s$ such that $(\ref{eq:cond2poly3})$ holds and $c_1,\ c_2>0 $.
   Therefore, Lemma \ref{lemma:projection} ensures that $P_A(t(a,b,0))\in
   C$. If $\lambda_1$ and $\lambda_2$ have distinct signs, then
   so do $c_1$ and $c_2$, and hence
   $P_A(t(a,b,0))\notin C$.

   Lastly, we show the case that
   $(a,b,0)\neq c(\alpha_1,\alpha_2,0)$ for any $c\neq 0$.
   If we write
   \[
   \begin{pmatrix}
    \alpha_1\\
    \alpha_2\\
    0
   \end{pmatrix}
  =d_1\begin{pmatrix}
	     -f_{1x}(0,0)\\
	     -f_{1y}(0,0)\\
	     1
	    \end{pmatrix}
  + d_2\begin{pmatrix}
	     -f_{2x}(0,0)\\
	     -f_{2y}(0,0)\\
	     1
	    \end{pmatrix}
  + \mu\begin{pmatrix}
	a\\
	b\\
	0
       \end{pmatrix},
   \]
   then $d_1$ and $d_2$ have distinct signs.
   Now
   \[
    c_1 =  \frac{1}{|M_0|}
     \begin{vmatrix}
     \alpha_1 & -f_{2x}(0,0) & a \\
     \alpha_2 & -f_{2y}(0,0) & b \\
     0 & 1 & 0
     \end{vmatrix}s + O(s^2)
   = d_1 s + O(s^2).
   \]
   Similarly we have $c_2 = d_2 s + O(s^2)$.
   Thus, for $t$ sufficiently close to $0$, we see that $c_1$ and $c_2$ have distinct signs.
   Therefore, Lemma \ref{lemma:projection} ensures the result.
  \end{proof}
  
     \begin{exmpl}
     \label{ex:region2}
    We consider $A=\{(x,y,z)\in \R^3:z \geq f_1(x,y),\ z \geq
    f_2(x,y)\}$, $B
    = \Span\{\nabla(z - f_1)(0,0,0), \nabla(z - f_2)(0,0,0)\}^\perp$ and $C = \{(x,y,z)\in \R^3:z =
    f_1(x,y)= f_2(x,y)\}$.
    
   \begin{enumerate}
    \item Let
	  \[
	   \begin{cases}
	    f_1(x,y) = x^2 + y^4,\\
	    f_2(x,y) = (x - 1)^2 + (y - 1)^4 - 2.
	   \end{cases}
	  \]
	  Then $\nabla (z-f_1)(0,0,0) = (0, 0, 1),\ \nabla(z -
	  f_2)(0,0,0) = (2, 4, 1)$, and $B = \{t(-2,1,0):t\in \R\}$. The curve $C$ is written as
	  \[
	   \begin{pmatrix}
	    x\\
	    y\\
	    z
	   \end{pmatrix}
	  =\begin{pmatrix}
	    -2y + 3y^2 - 2y^3\\
	    y \\
	    4y^2 - 12y^3 + 18y^4 - 12y^5 + 4y^6
	   \end{pmatrix}
	  = \begin{pmatrix}
	     -2 \\
	     1 \\
	     0
	    \end{pmatrix}y
	  + \begin{pmatrix}
	     3 \\
	     0 \\
	     4
	    \end{pmatrix} y^2 + O(y^3).
	  \]
   Now we have
   \[
    \begin{pmatrix}
     3 \\
     0 \\
     4
    \end{pmatrix}
   = \frac{37}{10}\begin{pmatrix}
      0 \\
      0 \\
      1
      \end{pmatrix}
   + \frac{3}{10} \begin{pmatrix}
	2\\
	4\\
	1
       \end{pmatrix}
   - \frac{6}{5}\begin{pmatrix}
       -2 \\
       1 \\
       0
      \end{pmatrix}.
   \]
   Then Theorem \ref{thm:cond2poly} guarantees that $P_A(t(-2,1,0))\in C$ for all
    sufficiently small $t\geq0$.

    \item Let
    \[
     \begin{cases}
      f_1(x,y) = x^2 + y^4,\\
      f_2(x,y) = \left(x + \frac{1}{2}\right)^2 + \left(y +
      \frac{1}{2}\right)^4 - \frac{5}{16}.
     \end{cases}
    \]
	  Then $\nabla (z-f_1)(0,0,0) = (0, 0, 1),\ \nabla(z -
	  f_2)(0,0,0) = (-1, -1/2, 1)$, $B = \{t(1,-2,0):t\in \R\}$. The curve $C$ is written as
	  \[
	   \begin{pmatrix}
	    x\\
	    y\\
	    z
	   \end{pmatrix}
	  =\begin{pmatrix}
	    -\frac{1}{2}y - \frac{3}{2}y^2 - 2y^3\\
	    y \\
	    \frac{1}{4}y^2 + \frac{3}{2}y^3 + \frac{21}{4}y^4 + 6y^5 + 4y^6
	   \end{pmatrix}
	  = \begin{pmatrix}
	     -\frac{1}{2} \\
	     1 \\
	     0
	    \end{pmatrix}y
	  + \begin{pmatrix}
	     -\frac{3}{2} \\
	     0 \\
	     \frac{1}{4}
	    \end{pmatrix} y^2 + O(y^3).
	  \]
   Now we have
   \[
    \begin{pmatrix}
     -\frac{3}{2} \\
     0 \\
     \frac{1}{4}
    \end{pmatrix}
   = -\frac{19}{20}\begin{pmatrix}
      0 \\
      0 \\
      1
      \end{pmatrix}
   + \frac{6}{5} \begin{pmatrix}
	-1\\
	-\frac{1}{2}\\
	1
       \end{pmatrix}
   -\frac{3}{10}\begin{pmatrix}
       1 \\
       -2 \\
       0
      \end{pmatrix}.
   \]
   By Theorem \ref{thm:cond2poly}, there exists $\epsilon
	  >0$ such that $P_A(t(1,-2,0))\notin C$ for 
    any $0 < t < \epsilon$.
	  
   \end{enumerate}    
   \end{exmpl}

\subsection{The Partition of the Region}
The boundary of $A$ consists of subsets of two surfaces
$\bA_1 = \{(x,y,z)\in \R^3:z = f_1(x,y) > f_2(x,y)\}$,
$\bA_2 = \{(x,y,z)\in \R^3:z = f_2(x,y) > f_1(x,y)\}$, and the curve $C =
\{(x,y,z)\in \R^3:z = f_1(x,y)=f_2(x,y)\}$.
  The $xy$-plane can be partitioned so that we can determine 
  to which part of $\partial A$ a point is mapped by $P_A$, as in Theorem $\ref{thm:region}$.
   We write $\{f_i - f_j * 0\} = \{(x,y)\in \R^2:f_i(x,y) - f_j(x,y)*0\}$
 for $i,j \in \{1,2\}$, where the symbol $*$ is $>,\ \geq$ or $=$.
 Let $\Psi_i:\R^2 \to \R^2$ be defined by
\[
  \Psi_i:\begin{pmatrix}
	x\\
	y
       \end{pmatrix}
 \mapsto
 \begin{pmatrix}
  x + f_{ix}(x,y)f_i(x,y)\\
  y + f_{iy}(x,y)f_i(x,y)
 \end{pmatrix},
\]
and $A_i=\{(x,y,z)\in \R^3:z \geq f_i(x,y)\}$ for $i = 1,2$.

   \begin{thm}
   \label{thm:region}
\begin{enumerate}
  \item   $P_A\circ\Psi_1(\{f_1 - f_2>0\})\subset \bA_1$.
 \item   $P_A\circ\Psi_2(\{f_2 - f_1>0\})\subset \bA_2$.
 \item   $P_A(\Psi_1(\{f_2 - f_1 \geq 0\})\cap \Psi_2(\{f_1 -
   f_2 \geq 0\})) \subset C$.
\end{enumerate}
 \end{thm}
  \begin{proof}
 For $(X,Y)\in \R^2$, Lemma \ref{lemma:APsystem} implies that
  $(x,y,f_1(x,y))=P_{A_1}(X,Y,0)$ if and only if
\[
  x + f_{1x}(x,y)f_1(x,y) = X,\ y + f_{1y}(x,y)f_1(x,y) = Y.
\]
  This means $(X,Y) = \Psi_1(x,y)$.
 Thus $\Psi_1$ is injective. Similarly, $\Psi_2$ is injective.
   
  Since $C$ is smooth, Lemma \ref{lemma:projection} gives that $(x,y,z)=P_A(X,Y,0)$ if and only if there exist
   $\lambda_1,\lambda_2\in \R$ such that
\begin{equation}
\begin{cases}
    &    -\begin{pmatrix}
     x - X \\
     y - Y \\
     z
	 \end{pmatrix}
  = \lambda_1\begin{pmatrix}
	      -f_{1x}(x,y,z)\\
	      -f_{1y}(x,y,z)\\
	      1
	     \end{pmatrix}
  + \lambda_2\begin{pmatrix}
	      -f_{2x}(x,y,z)\\
	      -f_{2y}(x,y,z)\\
	      1
 \end{pmatrix},\\
 & \lambda_1, \lambda_2 \leq 0,\quad
 \lambda_i(z - f_i(x,y)) = 0,\ i = 1,2.
\end{cases}
\label{eq:projection2}
\end{equation}   
  If $(X,Y)\in \Psi_1(\{f_1 - f_2>0\})$, then there exists $(x,y)$
  such that $f_1(x,y) > f_2(x,y)$
  and
  \[
   \begin{pmatrix}
    X\\
    Y \\
   \end{pmatrix}
  = \begin{pmatrix}
     x + f_{1x}(x,y)f_1(x,y)\\
     y + f_{1y}(x,y)f_1(x,y)
    \end{pmatrix}.
  \]
  Since $f_1(x,y)>0$, we have, for $\lambda_1 = -f_1(x,y)$,
  \[
     -\begin{pmatrix}
     x - X \\
     y - Y \\
     f_1(x,y)
    \end{pmatrix}
  = \lambda_1\begin{pmatrix}
	      -f_{1x}(x,y,z)\\
	      -f_{1y}(x,y,z)\\
	      1
	     \end{pmatrix}.
  \]
  Thus $(x,y,f_1(x,y))$ satisfies $(\ref{eq:projection2})$
  and hence $P_A(X,Y,0) = (x,y,f_1(x,y)) \in \bA_1$.
  Similarly, if $(X,Y)\in \Psi_2(\{f_2 - f_1 > 0\})$,
  then $P_A(X,Y,0) \in \bA_2$.
  Thus we have shown (i) and (ii).
  
  Next, we will show (iii).
  Let $(X,Y) \in \Psi_1(\{f_2 - f_1 \geq 0\})\cap \Psi_2(\{f_1 -
  f_2 \geq 0\})$ and $(x,y,z)=P_A(X,Y,0)$.
  Then the system $(\ref{eq:projection2})$ is satisfied for some $\lambda_1,\lambda_2\in \R$.
  If $(x,y,z) \in \bA_1$, then
  we have $z=f_1(x,y) > f_2(x,y)$, and hence
  \[
    -\begin{pmatrix}
     x - X \\
     y - Y \\
     z
    \end{pmatrix}
  = \lambda_1\begin{pmatrix}
	      -f_{1x}(x,y)\\
	      -f_{1y}(x,y)\\
	      1
	     \end{pmatrix}.
  \]  
   Since $\lambda_1=z=f_1(x,y)$, we have
   \[
    \begin{pmatrix}
     X\\
     Y
    \end{pmatrix}
   =
   \begin{pmatrix}
    x + f_1(x,y)f_{1x}(x,y)\\
    y + f_1(x,y)f_{1y}(x,y)
   \end{pmatrix}.
   \]  
  Thus $(X,Y)\in \Psi_1(\{f_1 - f_2>0\})$.
  Since $\Psi_1$ is injective, this contradicts to the inclusion $(X,Y) \in
  \Psi_1(\{f_2 - f_1 \geq 0\})$. Thus $(x,y,z) \in \partial A\setminus\bA_1$.
  Similarly, we have $(x,y,z)\in \partial A\setminus\bA_2$.  
  \end{proof}

 The boundary of $\Psi_1(\{f_1 - f_2>0\})$ is $\Psi_1(\{f_1 - f_2=0\})$,
 and the software \texttt{Macaulay2} \cite{GDSM} calculates its vanishing ideal 
\[
   \langle X - x - f_{1x}(x,y)f_1(x,y), Y - y - f_{1y}(x,y)f_1(x,y),\\
  f_1(x,y) - f_2(x,y)\rangle \cap \R[X,Y]
\]
by the elimination theory \cite{CLD2015} as in Figure \ref{fig1}.
 In the following examples, let $A = \{(x,y):z \geq f_1(x,y), z \geq f_2(x,y)\}$
 and $C  = \{(x,y,z):z=f_1(x,y)=f_2(x,y)\}$.
% \begin{figure}[ht]
%  \centering
%   \includegraphics[width=0.9\hsize]{region-crop.pdf} 
%         \caption{}
%         \label{fig1}
% \end{figure}

\begin{figure}[ht]
\begin{tabular}{c}
\begin{minipage}{0.45\hsize}
\centering
\begin{tikzpicture}[scale=0.75] % originally 0.85
\pgfplotsset{ticks=none}
  \begin{axis}
              [axis lines = center, xlabel = $x$,
            ylabel = $y$,
                xmax=0.3,
                xmin=-0.3,
                ymax=0.2,
                ymin=-0.2,
                legend style={nodes={scale=0.65, transform shape}}, 
                legend pos = south west, ]
                \addplot[ domain=-0.3:0.3, 
    samples=100, 
    color=black,
          very thin,
    ] plot {-0.5*x};
%    \addplot +[no markers, color=blue, ultra thick, 
%      raw gnuplot,
%      %empty line = jump % not strictly necessary, as this is the default behaviour in the development version of PGFPlots
%      ] gnuplot {
%      set contour base;
%      set cntrparam levels discrete 0.001;
%      unset surface;
%      set view map;
%      set isosamples 500;
%      set samples 500;
%      splot 256*x**9+2304*x**8*y+9216*x**7*y**2+21504*x**6*y**3+32256*x**5*y**4+32256*x**4*y**5+21504*x**3*y**6+9216*x**2*y**7+2304*x*y**8+256*y**9+1536*x**8+17664*x**7*y-576*x**6*y**2+263232*x**5*y**3+62592*x**4*y**4-147072*x**3*y**5-122352*x**2*y**6-39312*x*y**7-4152*y**8+2368*x**7-3520*x**6*y+139936*x**5*y**2-671680*x**4*y**3-68256*x**3*y**4+564144*x**2*y**5+233408*x*y**6+31024*y**7-9408*x**6-74976*x**5*y-448512*x**4*y**2-437760*x**3*y**3-580680*x**2*y**4-574008*x*y**5-128412*y**6-22544*x**5-184944*x**4*y-103000*x**3*y**2+1124632*x**2*y**3+37980*x*y**4+316340*y**5+3456*x**4+271152*x**3*y+463032*x**2*y**2+569088*x*y**3-383220*y**4+81828*x**3+292428*x**2*y+155232*x*y**2+142866*y**3+34020*x**2+105300*x*y+31995*y**2+14175*x+28350*y;
%    };
\addplot+[no markers, color=black, ultra thick, ] table [x=x, y=y,] {AltProj.pgf-plot1.table};
%    \addplot +[no markers, color=black, very thick,  dotted, 
%      raw gnuplot,
%      %empty line = jump % not strictly necessary, as this is the default behaviour in the development version of PGFPlots
%      ] gnuplot {
%      set contour base;
%      set cntrparam levels discrete 0.001;
%      unset surface;
%      set view map;
%      set isosamples 500;
%      set samples 500;
%      splot 256*x^9+2304*x^8*y+9216*x^7*y^2+21504*x^6*y^3+32256*x^5*y^4+32256*x^4*y^5+21504*x^3*y^6+9216*x^2*y^7+2304*x*y^8+256*y^9+1536*x^8-215808*x^7*y-297152*x^6*y^2-393792*x^5*y^3-330624*x^4*y^4-239872*x^3*y^5-97104*x^2*y^6-32880*x*y^7-1544*y^8+271424*x^7+2156224*x^6*y+2817376*x^5*y^2+3066496*x^4*y^3+1302368*x^3*y^4+963024*x^2*y^5+216048*x*y^6+20320*y^7-1707648*x^6-8001440*x^5*y-9973920*x^4*y^2-11083008*x^3*y^3-3096760*x^2*y^4-1538712*x*y^5-69892*y^6+3326384*x^5+11819184*x^4*y+17005896*x^3*y^2+21084600*x^2*y^3+4235428*x*y^4+455460*y^5-1700960*x^4-8042688*x^3*y-19171752*x^2*y^2-18292480*x*y^3-902688*y^4-551556*x^3+7917476*x^2*y+15352064*x*y^2+3117450*y^3+1064988*x^2-6647092*x*y-2091577*y^2+417907*x+835814*y;
%    };
\addplot+[no markers, color=black, very thick,  dotted, ] table [x=x, y=y,] {AltProj.pgf-plot2.table};

    \legend{$\pi(B)$, $\Psi_1(\{f_1-f_2=0\})$, $\Psi_2(\{f_2-f_1=0\})$}
  \end{axis}
\end{tikzpicture}
\end{minipage}
%\hspace{0.05\linewidth}
\begin{minipage}{0.45\hsize}
\centering
\begin{tikzpicture}[scale=0.75] % originally 0.85
\pgfplotsset{ticks=none}
  \begin{axis}
              [axis lines = center, xlabel = $x$,
            ylabel = $y$,
                xmax=0.2,
                xmin=-0.2,
                ymax=0.4,
                ymin=-0.4,
                legend style={nodes={scale=0.65, transform shape}}, 
                legend pos = south west, ]
                \addplot[ domain=-0.2:0.2, 
    samples=500, 
    color=black, very thin, 
    ] plot {-2*x};
%        \addplot +[no markers, color=blue, ultra thick, 
%      raw gnuplot,
%      empty line = jump % not strictly necessary, as this is the default behaviour in the development version of PGFPlots
%      ] gnuplot {
%      set contour base;
%      set cntrparam levels discrete 0.001;
%      unset surface;
%      set view map;
%      set isosamples 500;
%      set samples 500;
%      splot 4294967296*x^9+38654705664*x^8*y+154618822656*x^7*y^2+360777252864*x^6*y^3+541165879296*x^5*y^4+541165879296*x^4*y^5+360777252864*x^3*y^6+154618822656*x^2*y^7+38654705664*x*y^8+4294967296*y^9+42228252672*x^8+395657084928*x^7*y+1350138068992*x^6*y^2+2520797675520*x^5*y^3+2890184785920*x^4*y^4+2113947041792*x^3*y^5+972397215744*x^2*y^6+259717595136*x*y^7+30011293696*y^8+224245514240*x^7+1623359488000*x^6*y+4306750750720*x^5*y^2+6187189829632*x^4*y^3+5380532142080*x^3*y^4+2930617221120*x^2*y^5+896156172288*x*y^6+114821890048*y^7+616535922432*x^6+3227532298240*x^5*y+6422770836480*x^4*y^2+7381507534848*x^3*y^3+4989641351168*x^2*y^4+1868998938624*x*y^5+280793056256*y^6+921050271104*x^5+3557363556672*x^4*y+5249732093760*x^3*y^2+4903458227328*x^2*y^3+2451930720256*x*y^4+457519512576*y^5+923376487183*x^4+2403520574736*x^3*y+2500306295544*x^2*y^2+1905520236416*x*y^3+482175191856*y^4+678223892892*x^3+866137484936*x^2*y+816945658688*x*y^2+302531651232*y^3+192432299565*x^2+222370509788*x*y+102022852520*y^2+32796355804*x+16398177902*y;
\addplot+[no markers, color=black, ultra thick, ] table [x=x, y=y,] {AltProj.pgf-plot3.table};
%    };
%    \addplot +[no markers, color=black, very thick, dotted,  
%      raw gnuplot,
%      thick,
%      empty line = jump % not strictly necessary, as this is the default behaviour in the development version of PGFPlots
%      ] gnuplot {
%      set contour base;
%      set cntrparam levels discrete 0.001;
%      unset surface;
%      set view map;
%      set isosamples 500;
%      set samples 500;
%      splot 4194304*x^9+37748736*x^8*y+150994944*x^7*y^2+352321536*x^6*y^3+528482304*x^5*y^4+528482304*x^4*y^5+352321536*x^3*y^6+150994944*x^2*y^7+37748736*x*y^8+4194304*y^9-12582912*x^8-88080384*x^7*y-209387520*x^6*y^2-161316864*x^5*y^3+193597440*x^4*y^4+483836928*x^3*y^5+394666752*x^2*y^6+148137984*x*y^7+21322752*y^8+6701056*x^7-704512*x^6*y-134512640*x^5*y^2-303357952*x^4*y^3-55557120*x^3*y^4+287265024*x^2*y^5+217159040*x*y^6+44302528*y^7+5627904*x^6+71307264*x^5*y+132955392*x^4*y^2-47143296*x^3*y^3-25042032*x^2*y^4+84313956*x*y^5+37656891*y^6+17382400*x^5+125240064*x^4*y+205446656*x^3*y^2-33034208*x^2*y^3-67548948*x*y^4+794495*y^5-41084928*x^4-131517504*x^3*y-42408576*x^2*y^2-37647432*x*y^3-18338751*y^4+10020096*x^3-39187296*x^2*y-13354272*x*y^2-3822003*y^3+13387680*x^2+2786400*x*y+3184758*y^2+3425652*x+1712826*y;
%    };
    \addplot+[no markers, color=black, very thick,  dotted, ] table [x=x, y=y,] {AltProj.pgf-plot4.table};

    \legend{$\pi(B)$, $\Psi_1(\{f_1-f_2=0\})$, $\Psi_2(\{f_2-f_1=0\})$}
  \end{axis}
\end{tikzpicture}
\end{minipage}
\end{tabular}
\caption{Example~\ref{ex4.12} (left) and Example~\ref{ex4.13} (right)}\label{fig1}
\end{figure}
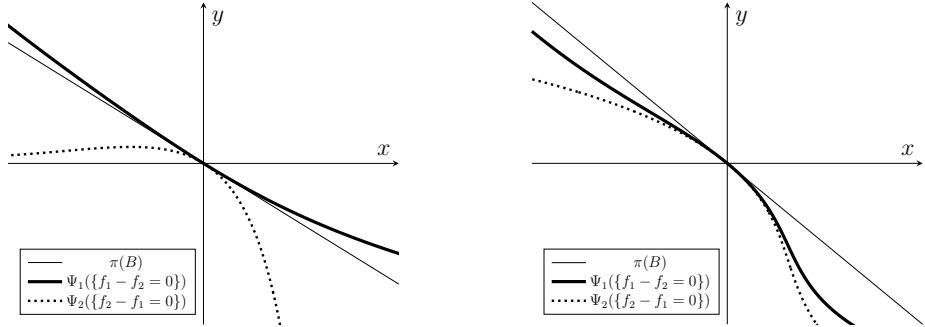

\begin{exmpl}\label{ex4.12}
  Let $f_1(x,y) = x^2 + y^4$ and $f_2(x,y) = (x - 1)^2 + (y - 1)^4 - 2$
 as in Example \ref{ex:region1}. 
 Let $B = \{(x,y,0):x + 2y = 0\}$; 
 the tangent line to the curve $C$ at $(0,0,0)$.
 In the left of Figure \ref{fig1}, the solid curve and the dotted curves are
 $\Psi_1(\{f_1 - f_2=0\})$ and $\Psi_2(\{f_2 - f_1=0\})$, respectively.
 The thin line is $\pi(B)$, where $\pi:(x,y,z)\mapsto (x,y)$.
 The line $\pi(B)$ passes through the origin while it keeps lying
 on the region $\Psi_1(\{f_2 - f_1 \geq 0\})\cap \Psi_2(\{f_1 -
 f_2 \geq 0\})$. Thus all points on $B$ are projected to $C$
 by $P_A$.

 Next, we will show that the convergence rate may differ depending on the initial point.
 We consider the projection of a point on $B' =
 \{t(0,1,0):t\in \R\}$ by $P_A$.
 Let $B'_+ = \{(0,y,0):y>0\}$ and $B'_- = \{(0,y,0):y<0\}$.
 For $(x,y) \in \Psi_1^{-1}(0,Y)$, we see that
 \begin{equation}
 x + 2x(x^2 + y^4) = 0,\ y + 4y^3(x^2 + y^4) = Y.
  \label{eq:region}
 \end{equation}
 Then we have $x = 0$. If $(0,Y,0)\in B'_+$, then the second equation of
 $(\ref{eq:region})$ ensures $y > 0$.
 Since $f_1(0,y) - f_2(0,y) = y^4 - ((y - 1)^4 - 1) = 2y(2(y -
 \frac{3}{4})^2 + \frac{1}{2}) > 0$,
 we have $(0,Y) \in \Psi_1(\{f_1 - f_2 > 0\})$.
 From (i) of Theorem \ref{thm:region}, a point in $B'_+$ is projected to $\bA_1 =
 \{(x,y,z):z = f_1(x,y) > f_2(x,y)\}$ by $P_A$.
 If the initial point is taken from $B'_+$,
 then the sequence $\{(x_k,y_k,0)\}$ constructed by alternating projections
 between $A$ and $B'$ behaves like those between $A_1 = \{(x,y,z):z \geq
       f_1(x,y)\}$ and $B'$.
 Since $f_1(t(0,1)) = t^4$, 
% Thus $\{(x_k,y_k)\}$ satisfies that $x_k = 0$ and $y_{k+1} + 4y_{k+1}^7 =
% y_k$.
       % Therefore, Lemma \ref{lemma:basic} 
       Theorem \ref{thm:hypersurface} gives $\|(x_k,y_k,0)\| = \Theta(k^{-\frac{1}{6}})$.
       
 Similarly, if the initial point is taken from $B'_-$,
 then the sequence $\{(x_k,y_k,0)\}$ constructed by alternating projections
 between $A$ and $B'$ behaves like those between $A_2 = \{(x,y,z):z \geq
 f_2(x,y)\}$ and $B'$.
 Here, the lowest degree of $f_2(t(0,1))=-4t + 6t^2 - 4t^3 + t^4$ is
       equal to $1$, and this means that $B'$ intersects transversely with $A_2$.
 Therefore, we see that $\{(x_k,y_k,0)\}$ converges linearly
 from a well-known fact that the alternating projection method
 for a transversal intersection converges linearly; see, e.g. \cite{GPR1967}, \cite{LM2008}.
\end{exmpl} 
\begin{exmpl}\label{ex4.13}
 Let  $f_1(x,y) = \left(x + \frac{1}{2}\right)^2 + \left(y +
 \frac{1}{2}\right)^4 - \frac{5}{16}$ and $f_2(x,y) = x^2 + y^4$ as in
	(ii) of Example \ref{ex:region2}.
	The tangent line to the curve $C$ at $(0,0,0)$ is given by $B = \{t(1,-2,0):t
	\in \R\}$.
  Now, the right of Figure \ref{fig1} corresponds to this case.
 We see that the thin line $\pi(B)$ is contained in $\Psi_1(\{f_1 - f_2 \geq 0\})$ around the origin.
 From (i) of Theorem \ref{thm:region}, any point on $\pi(B)$ close to the origin is projected to
 $\bA_1= \{(x,y,z):z=f_1(x,y)> f_2(x,y)\}$. 
% In addition, this case does not satisfy the sufficient condition in Theorem \ref{thm:cond2poly} as shown in Example \ref{ex:region2}.
 A sequence $\{(x_k,y_k,0)\}$ constructed by alternating projections
 between $A$ and $B$ behaves like those between $A_1 = \{(x,y,z):z \geq
	f_1(x,y)\}$ and $B$.
	Since $f_1(t(1,-2)) = 7t^2 - 16t^3 + 16t^4$,
	Theorem \ref{thm:hypersurface} gives $\|(x_k,y_k,0)\| = \Theta(k^{-\frac{1}{2}})$.
\end{exmpl}

\section{Intersections with Subspaces}
\label{sec:subsp}
We return to the case that a semialgebraic set $A$ is
defined by a single polynomial.
In Section \ref{section:hypersurface}, we have obtained
the exact convergence rate of the alternating projection method
if the intersecting subspace has a dimension one.
However, if the intersecting subspace has a dimension more than one, the
convergence rate depends on the initial point.
Section \ref{section:special} is devoted to a specific hypersurface to explain this phenomenon. 
In the remaining sections, we give upper bounds on the convergence rates, by applying the arguments for the exact rates.

\subsection{Hyperplanes}
We consider an intersection of a hypersurface $A$ defined by a convex
polynomial $g$ and a hyperplane $B$:
\begin{align*}
 A & = \{(x,z) \in \R^n\times \R:z \geq g(x)\},\\
 B & = \{(x,0): x \in \R^n\},
\end{align*}
where $A\cap B$ is a singleton.
We note that $g(x)>0\ (x\neq0)$ and $g(0) = 0$.

It is known that for any convergent power series
$f$ with $f(x)>0\ (x\neq 0),\ f(0) =
0$, there are $C>0$ and a rational number $\alpha\geq 1$ such that
\[
 f(x) \geq C\|x\|^\alpha\ \text{ for }x\text{ close to }0,
\]
see, e.g. \cite[equality $(G1)$]{KS2014}.
The smallest exponent $\alpha$ is called the \textit{\L{}ojasiewicz exponent} of $f$ and
denoted by $\cL(f)$.
\begin{exmpl}
\label{ex:loja}
 Corollary 2.1 of \cite{TB2014}
 says that if %$f(x)>0\ (x\neq0)$, $f(0)=0$ and
 $f$ is convenient and nondegenerate, then 
 $\cL(f)$ is the maximum length from the origin to the intersection of
 $\Gamma(f)$ and each axis.
 For example, if $f(x_1,x_2,x_3)=x_1^6 + x_2^4 + x_3^2$, we have $\cL(f) = 6$.
\end{exmpl}
We consider the sequences $\{a_k\},\ \{b_k\}$ constructed by
 alternating projections as
\[
 a_k \overset{P_B}{\longmapsto} b_k \overset{P_A}{\longmapsto} a_{k+1}.
\]
The following lemma is implied by
the inequality $(4.3)$ of \cite{BLY2014}.
We give a proof for the reader's convenience.
 \begin{lemma2}
  \label{lemma:convex_ineq}
 Let $A$ be a closed convex set, $B$ be a linear subspace and $A\cap
 B=\{0\}$.
 Then we have
 \[
\|a_{k+1}\|^2 +  d(a_k, B)^2 \leq \|a_k\|^2.
 \]
 \end{lemma2}
 \begin{proof}
  If $b_k = P_B a_k = 0$, then $d(a_k, B) = \|a_k\|$ and $a_{k+1} = P_A
  b_k = 0$.
  Thus we obtain the desired inequality. So, we assume $b_k \neq 0$.
 Since $B$ is a linear subspace, we have
  $\|a_k\|^2 = \|b_k\|^2 + \|b_k - a_k\|^2$.
  By the property of a projection, we see $b_k - a_{k+1} \in
  N_A(a_{k+1})$, which means
  $\langle b_k - a_{k+1}, a -  a_{k+1} \rangle \leq 0$ 
 for all $a\in A$.
 Since $0\in A$, we obtain $0 \geq \langle b_k - a_{k+1}, - a_{k+1} \rangle
 = -\langle b_k, a_{k+1} \rangle + \|a_{k+1}\|^2
 \geq -\|b_k\|\|a_{k+1}\| + \|a_{k+1}\|^2$.
  Thus we have $\|b_k\| \geq \|a_{k+1}\|$.
  Therefore $\|a_k\|^2 = \|b_k\|^2 + \|b_k - a_k\|^2 \geq \|a_{k+1}\|^2 + d(a_k,B)^2$. 
 \end{proof}
  \begin{thm}
   \label{thm:hshp}
  Suppose $g(x)^2\geq C\left(\sum_i x_i^2\right)^d$ for $x$ close to $0$. Then
  $b_k$ converges to $0$ in the rate $O\left(k^{\frac{-1}{2d-2}}\right)$.
  Moreover
   \[
  \limsup_{k\to \infty}\left((d-1)C\right)^{\frac{1}{2d-2}}k^{\frac{1}{2d-2}}\|b_k\|
 \leq 1. 
 \]
  \end{thm}
\begin{proof}
Let $a_k =(x_{1,k},\ldots, x_{n,k}, z_k)$ and $d_k := \|b_k\| = \|(x_{1,k},\ldots, x_{n,k})\|$.  Since $d(a_k, B) = z_k$, Lemma \ref{lemma:convex_ineq} gives
\[
\|(x_{1,k+1},\ldots, x_{n,k+1}, z_{k+1})\|^2 + z_k^2  \leq
  \|(x_{1,k},\ldots, x_{n,k}, z_k)\|^2.
  \]
  Thus we have
  \[
  d_{k+1}^2 + z_{k+1}^2 \leq d_k^2.
  \]
  Since $z_{k+1}^2=g(x_{1,k+1},\ldots,x_{n,k+1})^2\geq C\left(\sum_i x_{i,k+1}^2\right)$, we obtain
  \[
%  & d_{k+1}^2 + g(x_{1,k+1},\ldots,x_{n,k+1})^2 \leq d_k^2.
   d_{k+1}^2 + C(d_{k+1}^2)^d \leq d_k^2.
\]
 Since $\cL(g)\geq 1$, we have $2d=\cL(g^2)\geq 2$.
 By Corollary \ref{cor:basic}, we have
   \[
  \limsup_{k\to \infty}\left((d-1)C\right)^{\frac{1}{d-1}}k^{\frac{1}{d-1}}d_k^2
 \leq 1,
 \] 
 % $d_k^2 = O(k^{\frac{-2}{\ell-2}})$
 and thus $d_k = O\left(k^\frac{-1}{2d-2}\right)$.
\end{proof}
\begin{remark}
 Suppose that
  $A = \{(x,z)\in \R^n\times \R: z \geq g(x)\}$, 
  $B = \{(x,0)\in \R^n\times \R\}$
 and
$ (x_k,0) \overset{P_A}{\longrightarrow} (x,z) \overset{P_B}{\longrightarrow} (x_{k+1},0)$.
Then Lemma \ref{lemma:APsystem} implies that
 \begin{align*}
  x + g(x)\nabla g(x) & = x_k,\\
  x_{k+1} & = x.
 \end{align*}
 Thus the sequence $x_k$ is expected to follow the path defined by the gradient
 system
 \[
  \frac{d}{dt}x(t) = - \nabla \frac{1}{2}g^2(x(t)).
 \]
 The convergence rate of the gradient system is discussed in \cite[Thm
 1.6]{H2012}. The exponent used in their result can be obtained with $\cL(g^2)$ and is
 equal to
 the rate in this paper. 
\end{remark}

\subsection{Exact Rates for a Specific Polynomial}
\label{section:special}
 We consider the specific polynomial
 \[
  g(x,y) = x^2 + y^4.
 \]
  Let $A=\{(x,y,z)\in \R^3:z\geq g(x,y)\}$,
 $B = \{(x,y,0)\in \R^3:x,y \in \R\}$,
 and $b_k = \{(x_k,y_k,0)\}$ be the sequence constructed by $b_{k+1} = P_B\circ P_A(b_{k})$.
 Since $\cL(g) = 4$, Theorem \ref{thm:hshp} shows
 $\limsup\limits_{k\to \infty}Ck^\frac{1}{6}\|(x_k,y_k)\|\leq 1$,
 and thus the convergence rate has the upper bound $O(k^{-\frac{1}{6}})$.
 On the other hand, the following proposition gives \textit{exact} convergence rates, which depend on the
 initial points.
 Moreover, it shows that the exact rate achieves the upper bound
 for a generic initial point.
\begin{prop}
\label{prop:special}
 Let $\{(x_k,y_k)\}$ be the sequence defined by
 $(x_{k+1},y_{k+1},0) = P_B\circ P_A((x_k,y_k,0)$ for $k=0,1,\ldots$.
If $y_0 \neq 0$, then $(x_k,y_k)$ converges to $0$ in the exact rate
 of $\Theta(k^{-\frac{1}{6}})$.
 If $y_0=0$, then $(x_k,y_k)$
  converges to $0$ in the exact rate of $\Theta(k^{-\frac{1}{2}})$.
\end{prop}
 The proof uses the following two technical lemmas.
 By Lemma \ref{lemma:APsystem},
we have
\begin{equation}
  \begin{cases}
  x_{k+1}(1 + 2(x_{k+1}^2 + y_{k+1}^4)) = x_k\\
  y_{k+1}(1 + 4y_{k+1}^2(x_{k+1}^2 + y_{k+1}^4)) = y_k
  \end{cases},\ z_{k+1} = x_{k+1}^2 + y_{k+1}^4.
  \label{eq:x2y4}
\end{equation}
  \begin{lemma2}
   \label{lemma:x2y4}
  For sufficiently small $\epsilon>0$, if $0< x_k < y_k^2 \leq \epsilon$,
  then we have $0 < x_{k+1} < y_{k+1}^2 < \epsilon$.
  \end{lemma2}
 \begin{proof}
  Let $(X,Y)=(x_k,y_k)$, $(x,y)=(x_{k+1},y_{k+1})$.
 By $(\ref{eq:x2y4})$, we see that $x, y>0$ and $x\leq X,\ y\leq Y$.
 Now we have for sufficiently small $\epsilon>0$,
 \begin{align*}
  (1 + 4y^2(x^2 + y^4))^2 & \leq (1 + 4\epsilon(x^2 + y^4))^2\\
  &  = 1 + 8\epsilon (x^2 + y^4) + 16\epsilon^2 (x^2 + y^4)^2 \\
& = 1 + 2(x^2 + y^4) + 16\epsilon^2 (x^2 + y^4)^2 + (8\epsilon - 2) (x^2 +
  y^4)\\
%  & = 1 + 2(x^2 + y^4) + (x^2 + y^4)(16\epsilon^2 (x^2 + y^4) + 8\epsilon  - 2)\\
  & = 1 + 2(x^2 + y^4) + (x^2 + y^4)(32\epsilon^4 + 8\epsilon
  - 2)\\
  & \leq 1 + 2(x^2 + y^4).
 \end{align*}
Thus we obtain
 \[
 1 < \frac{Y^2}{X}
 = \frac{y^2}{x}\frac{(1 + 4y^2(x^2 + y^4))^2}{1 + 2(x^2 + y^4)}
 \leq \frac{y^2}{x}.
 \]
 \end{proof}
\begin{lemma2}
 \label{lemma:x2y4b}
 Suppose that $(x_0,y_0)$ be a point which is sufficiently
 close to $(0,0)$ and $x_0, y_0>0$.
 Then there exists $k_0$ such that $x_k < y_k^2$ for all $k>k_0$.
\end{lemma2}
  \begin{proof}
  By Lemma \ref{lemma:x2y4}, it is enough to show there exists $k_0$ such
   that $x_{k_0} < y_{k_0}^2$.  
   We show it by contradiction.
   Suppose that $x_k \geq y_k^2$ for all $k$.
   Since $x_k = x_{k+1}(1 + 2(x_{k+1}^2 + y_{k+1}^4))$, we have
   \[
    x_{k+1}(1 + 2x_{k+1}^2) \leq x_k %\leq x_{k+1}(1 + 4x_{k+1}^2).
   \]
   By Corollary \ref{cor:basic}, the inequality implies that
   $\displaystyle\limsup_{k\to \infty} 2k^\frac{1}{2}x_k \leq 1$.
   Then $x_k^2 \leq \frac{C}{k}$ for $C > \frac{1}{4}$ and sufficiently large $k$.
   
   Next, $y_k = y_{k+1}(1 + 4y_{k+1}^2(x_{k+1}^2 + y_{k+1}^4))$ gives that
   \[
y_{k+1}(1 + 4y_{k+1}^6)\leq y_k \leq y_{k+1}(1 + 8y_{k+1}^2 x_{k+1}^2).
   \]
   Since
   \[
   (1 + 8y_{k+1}^2x_{k+1}^2)(1 - 8y_k^2 x_{k+1}^2)
   = 1 - 8(y_k^2 - y_{k+1}^2)x_{k+1}^2 - 64y_{k+1}^2 y_k^2 x_{k+1}^4
   \leq 1,
   \]
   we obtain
   \begin{align*}
    y_{k+1} & \geq y_{k+1}(1 + 8y_{k+1}^2x_{k+1}^2)(1 - 8y_k^2
    x_{k+1}^2) \geq y_k(1 - 8y_k^2 x_{k+1}^2),\\
    y_{k+1}^2 & \geq y_k^2(1 - 8y_k^2 x_{k+1}^2)^2
    \geq y_k^2(1 - 16y_k^2 x_{k+1}^2),\\
    \frac{1}{y_{k+1}^2} & 
    \leq \frac{1}{y_k^2(1 - 16y_k^2 x_{k+1}^2)}
    \leq \frac{1}{y_k^2} + \frac{16x_{k+1}^2}{(1 - 16y_k^2 x_{k+1}^2)}.\\
   \end{align*}
   By summing the last inequality, we have
\begin{align*}
       \frac{1}{y_K^2} - \frac{1}{y_{K_0}^2} & \leq \sum_{k=K_0}^{K-1}
 \frac{16x_{k+1}^2}{(1 - 16y_k^2 x_{k+1}^2)}\\
 & \leq \sum_{k=K_0}^{K-1}\frac{16\frac{C}{k+1}}{1 - \frac{16C}{k+1}y_k^2}
 = \sum_{k=K_0}^{K-1}\frac{16C}{k+1 - 16Cy_k^2}.
\end{align*}
   Since $y_k\to 0$, for sufficiently large $K_0$, we see that $1 -
   16Cy_k^2>0$.
   Then we obtain
    \begin{align*}
     \frac{1}{y_K^2} - \frac{1}{y_{K_0}^2}
     & \leq \sum_{k=K_0}^{K-1}\frac{16C}{k+1 - 16Cy_k^2}\leq \sum_{k=K_0}^{K-1}\frac{16C}{k},\\
     \frac{1}{y_K^2} &  \leq 16C\log K + C_1,\\
     y_K^2\log K & \geq \frac{1}{16C + \frac{C_1}{\log K}} \geq C_2
    \end{align*}
   for some positive constants $C_1,C_2$.
   Thus, we have $x_k \leq \frac{C}{\sqrt{k}}$ and $y_k^2 \geq
   \frac{C_2}{\log k}$ for all sufficiently large $k$.
   This contradicts to the assumption that $x_k\geq y_k^2$ for all $k$.
\end{proof}

\begin{proof}
[Proposition \ref{prop:special}]
By symmetry, we may assume $x_0, y_0 \geq 0$.
 If $x_0, y_0>0$, then Lemma \ref{lemma:x2y4b} implies that for any $\epsilon>0$, we have $0<x_k < y_k^2\leq \epsilon$
 for all sufficiently large $k$. Then
 \[
  \|(x_k,y_k)\| \leq \sqrt{y_k^4 + y_k^2} = y_k\sqrt{1 + y_k^2}.
 \]
 Since $y_k = y_{k+1}(1 + 4y_{k+1}^2(x_{k+1}^2 + y_{k+1}^4))$, we have
 \[
  y_{k+1}(1 + 4 y_{k+1}^6) \leq y_k \leq y_{k+1}(1 + 8y_{k+1}^6).
 \]
 By Corollary \ref{cor:basic}, the first inequality implies
 $\limsup\limits_{k\to \infty} 24^\frac{1}{6}k^\frac{1}{6}y_k\leq 1$,\\
 and hence
  $\limsup\limits_{k\to \infty} 24^\frac{1}{6}k^\frac{1}{6}\|(x_k,y_k)\|\leq 1$.
 By similar arguments to Corollary \ref{cor:basic}, the second
 inequality
 implies $\liminf\limits_{k\to \infty}
 48^\frac{1}{6}k^\frac{1}{6}\|(x_k,y_k)\|\geq 1$.
 If $x_0=0, y_0>0$, then we have $x_k=0$ and
 $y_k = y_{k+1}(1 + 4y_{k+1}^6)$.
 Thus Lemma \ref{lemma:basic} implies that 
 $\lim\limits_{k\to \infty} 24^\frac{1}{6}k^\frac{1}{6}\|(x_k,y_k)\| = 1$.
  If $x_0>0, y_0=0$, then we have $y_k=0$ and
 $x_k = x_{k+1}(1 + 2x_{k+1}^2)$.
 Thus Lemma \ref{lemma:basic} implies
  $\lim\limits_{k\to \infty} 2 k^\frac{1}{2}\|(x_k,y_k)\| = 1$.
\end{proof}

\subsection{Subspaces with Dimensions More than One}
\label{sec:dim2}
We consider an intersection of a hypersurface $A$ defined by a convex
polynomial $g$ and a subspace $B$:
\begin{align*}
 A & = \{(x,z) \in \R^n\times \R:z \geq g(x)\},\\
 B & = \{(x,0)\in \R^n\times \R: x\in B_0\},
\end{align*}
where $B_0$ is an $r$-dimensional  subspace of $\R^n$, where $1\leq r \leq n-1$.
We assume that $g(x)>0\ (x\neq0)$ and $g(0) = 0$.

By rotation about $z$-axis, we may assume
\[
 B = \{(x,y,0)\in \R^{n-r}\times \R^r\times \R:x = 0\}.
\]
We consider the sequences $\{a(k)\},\ \{b(k)\}$ constructed by the alternating projections as
\begin{equation}
 a(k) \overset{P_B}{\longmapsto} b(k) \overset{P_A}{\longmapsto} a(k+1).
\label{eq:AP2}
\end{equation}
The following lemma is an easy consequence of Lemma 3.4 of
\cite{TB2014}.
  \begin{lemma2}
   \label{lemma:loja}
  Let $f(x)=\sum_{\alpha}f_\alpha x^\alpha \in \R\{x\}$
  and 
  %$f_\Gamma(x) = \sum_{\alpha \in\Gamma(f)}f_\alpha x^\alpha$
  $f_\Gamma(x) = \sum\{f_\alpha x^\alpha:\alpha \in\bigcup\Gamma(f)\cap \supp f\}$.  
 If $f$ is nonnegative and nondegenerate,
 then $\cL(f) = \cL(f_\Gamma)$.
  \end{lemma2}
\begin{thm}
\label{thm:loja}
 Suppose $g(0,y)^2$ is nondegenerate and
 $d = \cL(g(0,\cdot))$.
 Then
   $b(k)$ defined by \eqref{eq:AP2} converges to $0$ in the rate $O(k^{\frac{-1}{2d-2}})$.
\end{thm}
 \begin{proof}
  We write 
  \[
  a(k) = (x(k),y(k),z(k))= (x_1(k),\ldots,x_{n-r}(k),y_1(k),\ldots,y_r(k),z(k)).
  \]
  Then $b(k) =
  (0,y(k),0)$, and Lemma \ref{lemma:APsystem} gives 
   \begin{align}
&   x_i(k+1) + g_{x_i}(x(k+1),y(k+1))
   g(x(k+1),y(k+1)) = 0,\ i = 1,\ldots, n-r, \label{eq:subsp1}   \\
&   y_j(k+1) + g_{y_j}(x(k+1),y(k+1))
   g(x(k+1),y(k+1)) = y_j(k),\ j = 1,\ldots, r.
   \notag
   \end{align}
  Since $d(a(k),B)^2 = \|x(k)\|^2 + z(k)^2$ and $\|b(k)\|=\|y(k)\|$,
  Lemma \ref{lemma:convex_ineq} implies
  \[
  \|(x(k+1),y(k+1),z(k+1))\|^2 + \|x(k)\|^2 + z(k)^2 \leq \|(x(k),y(k),z(k))\|^2.
  \]
  In addition, since $z(k+1)=g(x(k+1),y(k+1))$, we obtain
  \[
  \|x(k+1)\|^2 + \|b(k+1)\|^2 + g(x(k+1),y(k+1))^2 \leq \|b(k)\|^2.
  \]
 %\begin{align*}
 % & \|(x(k+1),y(k+1),z(k+1))\|^2 + \|x(k)\|^2 + z(k)^2 \leq \|(x(k),y(k),z(k))\|^2, \\
 % & \|(x(k+1),y(k+1),z(k+1))\|^2 \leq \|y(k)\|^2,\\
 % & \|x(k+1)\|^2 + \|b(k+1)\|^2 + g(x(k+1),y(k+1))^2 \leq \|b(k)\|^2.
 %\end{align*}
  Here, we consider the system
\[
 x_{i} + g_{x_i}(x,y)
   g(x,y) = 0,\ i = 1,\ldots, n-r.
\]  
  By implicit function theorem,
  there exist convergent power series $\phi_i(y)$
  which solve equation $(\ref{eq:subsp1})$ as $x_i = \phi_i(y)$ and
  $\phi_i(0) = 0$ for $i
  = 1, \ldots, n-r$.
  We will apply Lemma \ref{lemma:newton}.
We claim that the Newton boundary of $g(0,y)$ meets all the axes.
In fact, if there exists $j$ such that the $j$th axis has no exponent of the
  support of $g(0,y)$, then $g(0,\ldots,0,y_j,0,\ldots,0) = 0$.
  It contradicts to $g(0,y)>0$ for $y\neq 0$.
  Since $g_{x_i}(0,0) = 0$, Lemma \ref{lemma:newton} implies that
  $I \subset \fm\fa$, where
  $I=\langle \phi_1(y),\ldots, \phi_{n-r}(y)\rangle$,
  $\fm = \langle y_1,\dots,y_r\rangle$, and $\fa = \langle
  y^\alpha:\alpha \in \supp(g(0,y))\rangle$.  

  Let $\phi(y) = (\phi_1(y),\ldots,\phi_{n-r}(y))$.
  Then we have
  \[
   \|b(k+1)\|^2 + g(\phi(y(k+1)),y(k+1))^2  + \|\phi(y(k+1))\|^2 \leq \|b(k)\|^2.
  \]
  Here, there exist polynomials $p_i$ such that
\begin{multline*}
    g(\phi(y),y)^2  = \left(g(0,y) +
 \sum\nolimits_{i=1}^{n-r}\phi_i(y)p_i(\phi(y),y)\right)^2\\
  = g(0,y)^2 +
 2g(0,y)\left(\sum\nolimits_{i=1}^{n-r}\phi_i(y)p_i(\phi(y),y)\right)+
 \left(\sum\nolimits_{i=1}^{n-r}\phi_i(y)p_i(\phi(y),y)\right)^2
\end{multline*}
  Since $I \subset \fm\fa$,
  we have $g(\phi(y),y)^2 - g(0,y)^2 \in \fa\fm\fa + \fm^2\fa^2 \subset  \fm\fa^2$.
  Thus the Newton boundary of $g(\phi(y),y)^2 + \|\phi(y)\|^2$
  is equal to that of $g(0,y)^2$.  
  Since $g(0,y)^2$ is nondegenerate, Lemma \ref{lemma:loja}
  implies that $g(\phi(y),y)^2 +
  \|\phi(y)\|^2$ and $g(0,y)^2$ have the same \L{}ojasiewicz exponent.
  Thus we obtain
  \[
   \|b(k+1)\|^2 + C\|b(k+1)\|^{2d} \leq \|b(k)\|^2
  \]
  for some $C>0$. By Corollary \ref{cor:basic}, we have $\|b(k)\| = O\left(k^{\frac{-1}{2d-2}}\right)$.
 \end{proof}

% \begin{remark}
%  \label{remark:loja}
%   If %$f(x)>0\ (x\neq0)$, $f(0)=0$ and
%  $f$ is convenient and nondegenerate, then Corollary 2.1 of \cite{TB2014}
%  says that
%  $\cL(f)$ is the maximum length from the origin to the intersection of
%  $\Gamma(f)$ and each axis.
%  For example, if $g(x_1,x_2,x_3)=x_1^6 + x_2^4 + x_3^2$, we have $\cL(g) = 6$.
% \end{remark} 

\begin{exmpl}
 Let $g=x_1^6 + x_2^4 + x_3^2$ as in Example \ref{ex:loja}, $A = \{(x,z)\in \R^3\times \R:z \geq g(x)\}$ and $B = \{(x,z)\in \R^3\times \R: x_1 = x_2 = 0\}$.
 Then $\cL(g(0,0,x_3))=2$ while $\cL(g) = 6$.
 Theorem \ref{thm:loja} implies that $b(k)$ defined by \eqref{eq:AP2} converges to $0$ in the rate $O(k^{\frac{-1}{2}})$.
\end{exmpl}

\section{Acknowledgements}
The first author was supported by JSPS KAKENHI Grant Number JP17K18726. The second author was supported by JSPS KAKENHI Grant Number JP19K03631. The third author was supported by JSPS KAKENHI Grant Number JP20K11696 and ERATO HASUO Metamathematics for Systems Design Project (No.JPMJER1603), JST.


\begin{thebibliography}{99}
   \bibitem{BB1996}
	  {\sc Bauschke, H.H., Borwein, J.M.},
	  {\em On projection algorithms for solving convex feasibility
	  problems},
	  SIAM Rev., 38 (1996), pp.~367--426.
  \bibitem{BCL2002}
	  {\sc Bauschke, H.H., Combettes, P. L., Luke, D.R.}, 
	  {\em Phase retrieval, error reduction algorithm, and Fienup variants;
	  a view from convex optimization}, 
	  Journal of the Optical Society of America, 
	  19 (2002), pp.~1334--1345.
  \bibitem{BLY2014}
	  {\sc Borwein, J.M., Li, G., Yao, L.},
	  {\em Analysis of the convergence rate for the cyclic projection
	  algorithm applied to basic semialgebraic convex sets},
	  SIAM, J. Optim, 24 (2014), pp.~498--527.
  \bibitem{C1996}
	  {\sc Combettes, P.L.}, 
	  {\em The convex feasibility problem in image recovery}, 
	  Advances in imaging and electron physics, 95 (1996), pp.~155--270. 
  \bibitem{CLD2015}
	  {\sc Cox, D.A., Little, J., O'Shea, D.},
	  {\em Ideals, Verieties, and Algorithms, 4th edition},
	  Springer-Verlag, 2015.
  \bibitem{DLW2017}
	  {\sc Drusvyatskiy, D., Li, G., Wolkovicz, H.},
	  {\em A note on alternating projections for ill-posed semidefinite
	  feasibility problems},
	  Math. Program. 162 (2017), pp.~537--548.
  \bibitem{GDSM}
	  {\sc Grayson, D.R., Stillman, M.E.},
          {\em Macaulay2, a software system for research in algebraic geometry},
          Available at \url{http://www.math.uiuc.edu/Macaulay2/}
  \bibitem{GS1996}
	  {\sc Grigogiadis, K.M., Skelton, R.E.},
	  {\em Low-order control design for LMI problems using
	  alternating projection methods},
	  Automata, 32 (1996), pp.~1117--1125. 	            
  \bibitem{GD2020}
	  {\sc Groetzner, P., D\"ur, M.}, 
	  {\em A factorization method for completely positive matrices},
	  Linear Algebra Appl., 591 (2020), pp.~1--24.	  
  \bibitem{GPR1967}
	  {\sc Gubin, L.G., Polyak, B.T., Raik, E.V.},
	  {\em The method of projections for finding the common point of convex sets},
	  Comput. Math. Math. Phys., 7 (1967), pp.~1--24.	  
  \bibitem{H2012}
	  {\sc Haraux, A.},
	  {\em Some applications of the \L{}ojasiewicz gradient inequality},
	  Comm. Pure Appl. Anal., 11 (2012), pp.~2417--2427.
  \bibitem{KP2013}
	  {\sc Krantz, S.G., Parks, H.R.},
	  {\em The implicit function theorem}, Birkhauser Verlag, 1992.  
  \bibitem{KLT2018}
      {\sc Kruger, A.Y., Luke, D.R., Thao, N. H.},
      {\em Set regularities and feasibility problems},
      Math. Program. 168 (2018), pp.~279--311.
	  Springer-Verlag, 2013.
  \bibitem{KS2014}
	  {\sc Kurdyka, K., Spodzieja, S.},
	  {\em Separation of real algebraic sets and the \L{}ojasiewicz
	  exponent},
	  Proc. Amer. Math. Soc., 142 (2014), pp.~3089--3102.
  \bibitem{LM2008}
	  {\sc Lewis, A. S., Malick, J.},
	  {\em Alternating Projections on Manifolds},
	  Math. Oper. Res., 33 (2008), pp.~216--234.
  \bibitem{LTT2018}
      {\sc Luke, D.R., Thao, N.H., Tam, M.K.},
      {\em Quantitative Convergence Analysis of Iterated Expansive Set-Valued Mappings},
      Math. Oper. Res., 43 (2018), pp.~1143--1176.
  \bibitem{RW1998}
	  {\sc Rockafellar, R.T., Wets, R.J.-B.},
	  {\em Variational analysis}, Springer-Verlag, 1998.
  \bibitem{TB2014}
	  {\sc Thao, N., Bui, N.},
	  {\em Computation of the \L{}ojasiewicz exponent of nonnegative and
	  nondegenerate analytic functions},
	  Internat. J. Math., 25 (2014), 1450092.
 \end{thebibliography}
\end{document}